\documentclass[10pt,reqno]{amsart}

\usepackage[utf8]{inputenc}
\usepackage[english]{babel}
\usepackage{amsmath,amsfonts,amssymb,amsthm,shuffle}
\usepackage[T1]{fontenc}
\usepackage{lmodern}
\usepackage{mathtools}
\usepackage[titletoc]{appendix}

\usepackage[top=3.5cm,bottom=3.5cm,left=3.6cm,right=3.6cm]{geometry}

\usepackage[dvipsnames]{xcolor}
\usepackage[hyperindex=true,frenchlinks=true,colorlinks=true,
citecolor=Mahogany,linkcolor=DarkOrchid,urlcolor=Tan,linktocpage,
pagebackref=true]{hyperref}

\usepackage{tikz}
\usetikzlibrary{shapes}
\usetikzlibrary{fit}
\usetikzlibrary{decorations.pathmorphing}

\usepackage{dsfont}
\usepackage{wasysym}
\usepackage{stmaryrd}
\usepackage{cite}
\usepackage{subfigure}
\usepackage{multirow}
\usepackage{enumitem}
\usepackage{multicol}
\usepackage{amsaddr}

\title[]{Limit theorems for U-statistics of Bernoulli data}
\author{Davide Giraudo  }
 
\address{Davide.Giraudo@rub.de \\  Ruhr-Universität Bochum, Germany}
\keywords{Law of large numbers; law of the iterated logarithms; central 
limit theorem; U-statistics; functionals of i.i.d.}
\date{\today}

\numberwithin{equation}{section}
\renewcommand{\leq}{\leqslant}
\renewcommand{\geq}{\geqslant}

\newtheorem{Theorem}{Theorem}[section]
\newtheorem{Th\'eor\`eme}{Th\'eor\`eme}[section]
\newtheorem{Proposition}[Theorem]{Proposition}
\newtheorem{Lemma}[Theorem]{Lemma}

\newtheorem{D\'efinition}[Th\'eor\`eme]{D\'efinition}
\newtheorem{Corollary}[Theorem]{Corollary}

\theoremstyle{remark}
\newtheorem{Remark}[Theorem]{Remark}

\tikzstyle{Vertex}=[circle,draw=LimeGreen!80,fill=LimeGreen!8,
inner sep=1pt,minimum size=2mm,line width=1pt,font=\scriptsize]
\tikzstyle{Node}=[Vertex,draw=RoyalBlue!80,fill=RoyalBlue!8,inner sep=1.5pt]
\tikzstyle{Leaf}=[rectangle,draw=Black!70,fill=Black!16,
inner sep=0pt,minimum size=1mm,line width=1.25pt]
\tikzstyle{Edge}=[Maroon!80,cap=round,line width=1pt]
\tikzstyle{Mark1}=[draw=BrickRed!80,fill=BrickRed!8]
\tikzstyle{Mark2}=[draw=BurntOrange!80,fill=BurntOrange!8]
\tikzstyle{EdgeRew}=[->,RedOrange!80,cap=round,thick]

\newcommand{\Fca}{\mathcal{F}}

\newcommand \ens[1]{\left\{ #1\right\}}
\newcommand \R{\mathbb R}

\newcommand \PP{\mathbb P}
\newcommand{\el}{\mathbb L}
\newcommand{\E}[1]{\mathbb E\left[#1\right]}

\newcommand \Z{\mathbb Z}

\newcommand \abs[1]{\left|#1\right|}
\newcommand \eps{\varepsilon}

\newcommand{\f}{\mathcal F}

\newcommand{\pr}[1]{\left(#1\right)}
\newcommand{\norm}[1]{\left\lVert #1 \right\rVert}

\newcommand{\ent}[1]{\left\lfloor #1\right\rfloor}
\newcommand{\ind}{\operatorname{ind}}
\newcommand{\til}[1]{\widetilde{#1}}
\newcommand{\LL}[1]{\operatorname{LL}\pr{#1}}

\begin{document}


\begin{abstract}
In this paper, we consider U-statistics whose data is 
a strictly stationary sequence which can be expressed as a functional of an 
i.i.d. one. We establish a strong law of large numbers, a bounded law of the 
iterated logarithms and a central limit theorem under a dependence condition. The main 
ingredients for the proof are an approximation by U-statistics whose data is a functional 
of $\ell$ i.i.d. random variables and an analogue of the Hoeffding's decomposition 
for U-statistics of this type.
\end{abstract}
\maketitle

 \section{Introduction and main results}
 
 \subsection{Context}\label{subsec:context}
 
Let $\pr{X_j}_{j\geq 1}$ be a strictly stationary sequence, in the sense 
that the vectors $\pr{X_i}_{i=1}^n$ and  $\pr{X_{i+k}}_{i=1}^n$ have the 
same distribution for all $n$ and $k\geq 1$. The U-statistic of kernel 
$h\colon\R\times \R\to \R$ and data $\pr{X_j}_{j\geq 1}$ is defined as 
\begin{equation}
U_n:=\sum_{1\leq i<j\leq n}h\pr{X_i,X_j}, \quad n\geq 2.
\end{equation}

The study of the asymptotic behavior of the sequence $\pr{U_n}_{n\geq 2}$ 
properly normalized is a question of interest in probability theory and the 
applications. We will be interested in the following three limit theorems. 

\begin{enumerate}
\item Law of large numbers: let $1\leq p<2$; the following convergence holds 
\begin{equation}\label{eq:definition_LGN}
\frac 1{ n^{1+1/p}}\pr{ 
 U_n -\E{ U_n}
}\to 0 \mbox{ a.~s.};
\end{equation}
\item Bounded law of the iterated logarithms: the random variable 
\begin{equation}\label{eq:definition_LLI}
\sup_{n\geq 1}\frac 1{n^{3/2}\sqrt{ \LL{n} }}\abs{
 U_n -\E{ U_n }
 }
\end{equation}
is almost surely finite, where $L\colon \R_+\to\R_+$ is defined by 
$L\pr{x}=\max\ens{\ln x,1}$ and $\LL{x}:=L\circ L\pr{x}$.

\item Central limit theorem:  there exists a $\sigma>0$ such that 
\begin{equation}\label{eq:TLC_iid}
\frac 1{n^{3/2}}\pr{ 
 U_n -\E{ U_n}}\to \sigma N\mbox{ in distribution},
\end{equation}
where $N$ is a standard normal random variable.

\end{enumerate} 
 
Usually, the conditions for guaranting this kind of limit theorems 
are on the dependence of the sequence $\pr{X_j}_{j\geq 1}$ and 
also on the kernel $h$, for example by requiring some integrability conditions 
on $h\pr{X_1,X_2}$.
We will  first review a few results for the case where the data $\pr{X_j}_{j\geq 1}$ 
is i.i.d. and the kernel $h$ is symmetric.

\begin{enumerate}
\item If $1\leq p<2$ and $h\pr{X_1,X_2}\in\mathbb L^p$, then 
\eqref{eq:definition_LGN} holds \cite{MR1227625}.

\item If  $h\pr{X_1,X_2}\in\mathbb L^2$, then the 
random variable defined by \eqref{eq:definition_LLI} is almost surely finite. Moreover, 
for all $1<p<2$, according to Theorem~2.5 in \cite{MR1348376}, the 
following inequality holds:
\begin{equation}
\norm{\sup_{n\geq  1}\frac 1{\sqrt{nL\pr{L\pr{ n}}}}\abs{
 U_n\pr{h,f,\pr{\eps_i}_{i\in \Z}  }-\E{ U_n\pr{h,f,\pr{\eps_i}_{i\in \Z}  }}
 }}_p\leq C_p \norm{h\pr{X_1,X_2}}_2.
\end{equation}
 
\item The convergence \eqref{eq:TLC_iid} has been established in 
\cite{MR0026294}.
\end{enumerate}

The extension of these results to the case of stationary dependent data is 
a challenging problem. The law of large numbers has been established under 
a $\beta$-mixing assumption in \cite{MR1624866,MR2571765}. In the case $p=1$, the 
independence assumption can be replaced by a $4$-dependence assumption, that is, 
for all $4$-uple of distinct integers $\pr{i_1,i_2,i_3,i_4}$, the collection of random variables $\pr{X_{i_k}}_{k=1}^4$ is independent. If the sequence $\pr{X_j}_{j\geq 1}$
is identically distributed, then the law of large numbers hold (Theorem~1 in 
\cite{MR1650528}). The question of an equivalent of the ergodic theorem 
when the data is any strictly stationary sequence was considered in 
\cite{MR1363941}.
The question of the law of the iterated logarithms was also adressed in 
\cite{MR2731072} where the data is allowed to be $\alpha$-mixing or 
a functional of a $\beta$-mixing sequence.
A central limit theorem has been established under a $\beta$-mixing 
assumption in \cite{MR1353588} and under an $\alpha$-mixing 
condition in \cite{MR2557623}. The $\tau$-dependent case 
has been addressed in \cite{MR2922461}. The case of 
associated random variables has also been investigated in 
\cite{MR1911807}. Note that there are also results 
dealing with the asymptotic behavior of $U$-statistics of 
a semimartingale \cite{MR3262509} and of Poisson processes 
\cite{MR3161465}.
 
In this paper, we will be interested in establishing the law of large number, 
the law of the iterated logarithms and the central limit theorem when the 
data can be expressed as a functional of an i.i.d. sequence. In a similar 
context on the data but for weighted U-statistics, the central
 limit theorem was investigated in \cite{MR2060311}.

Let us precise the context. Let $\pr{\eps_i}_{i\in \Z}$ 
 be an independent identically distributed sequence with values in $\R^k$, $k\geq 1$. 
 Given  measurable functions $h\colon \R^k\times \R^k\to \R$ and 
 $f\colon \pr{\R^k}^{\Z}  \to \R$, we are interested in the 
 asymptotic behavior of the U-statistic of order two defined by 
 \begin{equation}
 U_n\pr{h,f,\pr{\eps_i}_{i\in \Z}  }=
 \sum_{1\leq i<j\leq n}h\pr{ 
 f\pr{\pr{\eps_{i-k}}_{k\in \Z}   },  f\pr{\pr{\eps_{j-k}}_{k\in \Z}   }
 },
 \end{equation}
that is,  letting $X_j:=  f\pr{\pr{\eps_{j-k}}_{k\in \Z}}$, $U_n$ is a U-statistic 
of kernel $h$ and the data is the strictly stationary sequence 
$\pr{X_j}_{j\geqslant 1}$. More precisely, we are interested in 
conditions involving the kernel $h$ and the sequence $
\pr{f\pr{\pr{\eps_{j-k}}_{k\in \Z}}}_{j\geq 1}$
which guarantee the previously mentioned limit theorems.

In all the paper, the kernel $h$ is supposed to be symmetric in the sense that 
$h\pr{x,y}=h\pr{y,x}$ for all $x,y\in\R^k$. 
 
The paper is organized as follows. In Subsection~\ref{subsec:measure_of_dependence}, 
we will introduce a measure of dependence of a U-statistic whose data 
is a functional of an i.i.d. sequence. In Subsection~\ref{subsec:generalized_Hoeffding}, 
we formulate an analogue of the Hoeffding decomposition for such U-statistics. 
Subsections~\ref{subsec:LLN},~\ref{subsec:LLI} and~\ref{subsec:TLC}
are devoted respectively to the statements of the law of large numbers, 
the bounded law of the iterated logarithms and the  central limit theorems 
for U-statistics of Bernoulli data. In Subsection~\ref{subsec:applications}, 
we give examples of kernels $h$ for which the measure of dependence 
can be estimated only with the help of the dependence of the data.
Section~\ref{sec:proofs} is devoted to the proofs of the previously mentioned 
results.

 \subsection{Measure of dependence}\label{subsec:measure_of_dependence}
 
 The  extension of the results of the i.i.d. case requires 
 a measure of dependence. Let 
 $\pr{\eps_u}_{u\in\Z}$ be an i.i.d. sequence, $f\colon \R^\Z\to \R$ a 
 measurable function and $h\colon\R\times\R\to\R$. In order to deal 
 with the dependence which comes into play in $U_n\pr{h,f,
 \pr{\eps_i}_{i\in\Z}}$, we need the following notations. Denote 
 $X_j:=f\pr{\pr{\eps_{j-i}}_{i\in\Z}}$ and define the random vectors 
 \begin{equation}
  V_{j,\ell}:=\pr{\eps_u}_{u=j-\ell}^{j+\ell}.
 \end{equation}
  The random variable 
  $\E{X_j\mid  V_{j,\ell}}$ can be writen as a function of 
  $V_{j,\ell}$ and by stationarity and 
  Lemma~\ref{lem:propriete_esperance_cond_fonction_d_iid}, the 
  involved function does not depend on $j$. Therefore, we write 
  \begin{equation}
   \E{X_j\mid  V_{j,\ell}}=f_\ell\pr{V_{j,\ell}}.
  \end{equation}
  We then define for $p\geq 1$ and $\ell\geq 1$ the 
  coefficient of dependence 
  \begin{equation}\label{eq:definition_de_la_mesure_de_dependence}
   \theta_{\ell,p}:= 
   \sup_{j\geq 0}\norm{
   h\pr{f_\ell\pr{V_{0,\ell}},f_\ell\pr{V_{j,\ell}}  }
   -h\pr{f_{\ell-1}\pr{V_{0,\ell-1}},f_{\ell-1}\pr{V_{j,\ell-1}}  }
   }_p
  \end{equation}
  and for $\ell=0$, 
  $\theta_{0,p}= \sup_{j\geq 0}\norm{h\pr{f_0\pr{V_{j,\ell}},
  f_0\pr{V_{0,\ell}}}}_p$.
  
  In particular, finiteness of $\sum_{\ell\geq 1}\theta_{\ell,p}$
  allows to write 
  \begin{multline}\label{eq:decompositio_avec_f_l}
   U_n\pr{h,f,
 \pr{\eps_i}_{i\in\Z}}=\sum_{1\leq i<j\leq n}
 h\pr{f_0\pr{V_{i,0}},f_0\pr{V_{j,0}}}\\
 +
 \sum_{\ell\geq 1}\sum_{1\leq i<j\leq n}
  h\pr{f_\ell\pr{V_{i,\ell}},f_\ell\pr{V_{j,\ell}}  }
   -h\pr{f_{\ell-1}\pr{V_{i,\ell-1}},f_{\ell-1}\pr{V_{j,\ell-1}}  },
  \end{multline}
  where the convergence takes place in $\mathbb L^p$ and almost 
  surely. The interest of the decomposition \eqref{eq:decompositio_avec_f_l}
  is that it reduces the treatmeant of the original U-statistic to that of 
  U-statistics whose data is a strictly stationnary sequence 
  which is a functional of $2\ell+1$ independent identically distributed 
  random variables. Nevertheless, this task requires some work 
  in order to be reduced to U-statistics of independent data.

 \subsection{A generalized Hoeffding's decomposition}
\label{subsec:generalized_Hoeffding} 
 
 A usefull tool to establish limit theorems for U-statistics with i.i.d. data
 is the Hoeffdings's decomposition \cite{MR0026294}. 
Let $h\colon\R^k\times \R^k\to \R$ be a symmetric 
measurable function and $\pr{X_j}_{j\in 
\Z}$ be an i.i.d. sequence. 
 We write decompose $h$ in the following way:
 \begin{equation}
 h\pr{x,y}=\theta+h_1\pr{x}+h_1\pr{y}+h_2\pr{x,y},
 \end{equation}
 where $\theta=\E{h\pr{X_1,X_2}}$, 
 \begin{equation}
 h_1\pr{x}=\E{h\pr{X_1,x}}-\theta\mbox{ and }
 \end{equation}
 \begin{equation}
 h_2\pr{x,y}=h\pr{x,y}-h_1\pr{x}-h_1\pr{y}-\theta.
 \end{equation}
 In this way, the following equality holds
 \begin{equation}\label{eq:Hoeffding_iid}
 \sum_{1\leq i<j\leq n}h\pr{X_i,X_j}= 
 \binom n2\theta+n\sum_{i=1}^nh_1\pr{X_i}+
 \sum_{1\leq i<j\leq n}h_2\pr{X_i,X_j}.
 \end{equation}
 The part involving $h_2$ can be treated by martingale 
 techniques, since the sequence $\pr{\sum_{i=1}^{j-1}
 h_2\pr{X_i,X_j}}_{j\geq 1}$ is a martingale differences 
 sequence with respect to the filtration $\pr{\sigma\pr{X_u,u\leq j}}$ 
 and the terms $\sum_{i=1}^{j-1}
 h_2\pr{X_i,X_j}$ can be treated thanks to a reverse martingale differences 
 property.
 
We would like to extend this to the setting mentioned in 
\ref{subsec:context}, that is, $X_j:=  f\pr{\pr{\eps_{j-k}}_{k\in \Z}}$, where 
$\pr{\eps_u}_{u\in \Z}$ is an i.i.d. sequence. We introduce the following 
notation 
\begin{equation}
U_n^{\ind}\pr{h,\pr{\eps_i  }_{i\in \Z}}
=\sum_{1\leq i<j\leq n}h\pr{\eps_i ,\eps_j   },
\end{equation}
where $\pr{\eps_i  }_{i\in \Z}$ is an i.i.d. sequence of random variables 
with values in $\R^k$ and $h\colon\R^k\times \R^k\to \R$ is a measurable 
function.

One naturally expects the decomposition to be more complicated than
in the independent case. Let us point out the major differences and common 
points. Like in the independent case, the decomposition of the
 centered U-statistic involves a stationary sequence and a degenerated U-statistic. 
 But in the context of Bernoulli data, one get a series involving stationary 
 sequences and a series of degenerated U-statistics. An other difference 
 is that we also have remainder terms which are not directly associated 
 to the involved stationary sequences or degenerated U-statistic. 
 The origin of these terms will be explained during the proof.

\begin{Proposition}\label{prop:Hoeffding_decomposition}
Let $\pr{\eps_u}_{u\in \Z}$ be an i.i.d. sequence of random variables, 
$\pr{\eps'_u}_{u\in \Z}$ an independent copy of $\pr{\eps_u}_{u\in \Z}$ and 
$f\colon \R^{\Z}\to \R$ be a measurable function.
Let $V_{k,\ell}:=\pr{\eps_{u}}_{u=k-\ell}^{k+\ell}$ and 
$V'_{k,\ell}:=\pr{\eps'_{u}}_{u=k-\ell}^{k+\ell}$.
 Let $f_\ell\colon 
\R^{2\ell+1}\to \R$ be a function such that 
\begin{equation}
f_\ell\pr{V_{k,\ell}  }= 
\E{ f\pr{ \pr{\eps_{k-u}   }_{u\in \Z}      }   \mid 
V_{k,\ell} }\mbox{ a.s.}
\end{equation}
Let $h\colon\R\times \R\to \R$ be a symmetric measurable function. 
Assume that the following convergence holds almost surely for all 
$1\leq i<j$:
\begin{equation}\label{eq:convergence_presque_sure_pour_decomposition}
\lim_{\ell\to +\infty}  
h\pr{  f_\ell\pr{V_{i,\ell}  },f_\ell\pr{V_{j,\ell}  }    }
= h\pr{f\pr{  \pr{\eps_{i-u}}_{u\in\Z}    },
f\pr{  \pr{\eps_{j-u}}_{u\in\Z}     }
}.
\end{equation}
Then the following equality holds:
\begin{multline}
U_n\pr{h,f,\pr{\eps_i}_{i\in\Z}}-\E{U_n\pr{h,f,\pr{\eps_i}_{i\in\Z}}}=n\sum_{k=1}^n\E{h\pr{f_0\pr{V_{k,0}},
f_0\pr{V'_{k,0}}  }\mid V_{k,0}}+ 
U^{\ind}_n\pr{h^{\pr{0}},\pr{\eps_i}_i}+
 \\
 \sum_{\ell \geq1} 
\pr{4\ell+1} \ent{\frac{n}{4\ell+1}}
 \sum_{k=1}^{\pr{4\ell+2} \ent{\frac{n}{4\ell+2}}+1 }
\E{h\pr{f_\ell\pr{V_{k,\ell}},f_{\ell}\pr{  V'_{0,\ell}  }} \mid V_{k,\ell}  }
-\E{h\pr{f_{\ell-1}\pr{V_{k,\ell-1}},f_{\ell-1}\pr{  V'_{0,\ell-1}  }} \mid V_{k,\ell-1}  }
 \\
+\sum_{\ell\geq 1}\sum_{a,b\in [4\ell+2]} U_{\ent{\frac{n}{2\ell}}   }^{\ind}\pr{h_{a,b}^{\pr{\ell}}    ,
\pr{\eps_i^{a,b}}   }+R_{n,1,1}+R_{n,1,2}+
\sum_{k=2}^6R_{n,k},
\end{multline}
where for each $\ell\geq 1$ and all  $a,b\in [4\ell+2]$, 
the U-statistic  $U_{\ent{\frac{n}{2\ell}}   }^{\ind}\pr{h_{a,b}^{\pr{\ell}}    
,\pr{\eps_i^{a,b}}   }$ has independent data and 
is degenerated, and the remainder terms are defined as 
\begin{equation}
R_{n,1,1}:= \sum_{\ell\geq 1}\sum_{j=\pr{4\ell+2}\ent{\frac{n}{4\ell+2}   }+1     }^n 
\sum_{i=\pr{4\ell+2}\ent{\frac{j-1}{4\ell+2}   }+1 }^{j-1}H^{\pr{\ell}}_{i,j}
\end{equation}
\begin{equation}
R_{n,1,2}:= \sum_{\ell\geq 1}
\sum_{a=1}^{4\ell+2}\sum_{j=\pr{4\ell+2}\ent{\frac{n}{4\ell+2}   }+1     }^n 
\sum_{k=0}^{\ent{\frac{j-1}{4\ell+2}   }-1}H^{\pr{\ell}}_{\pr{4\ell+2}k+a,j};
\end{equation}
\begin{equation}
R_{n,2}:=\sum_{\ell\geq 1}\sum_{u=0}^{\ent{\frac{n}{4\ell+2}}-1}
\sum_{\substack{ a,b\in [ 4\ell+2]\\ a<b  }   }
H^{\pr{\ell}}_{u\pr{4\ell+2}+a,u\pr{4\ell+2}+b}
\end{equation}
 \begin{equation}
R_{n,3}:= \sum_{\ell\geq 1}\sum_{v=1}^{\ent{\frac{n}{4\ell+2}}-1}
\sum_{\substack{ a,b\in [ 4\ell+2]\\ 0\leq a-b\leq \pr{2\ell+1}-1   }   }
\pr{H^{\pr{\ell}}_{ a,v\pr{4\ell+2}+b}+H^{\pr{\ell}}_{ b,v\pr{4\ell+2}+a}}
\end{equation}
 \begin{equation}
R_{n,4}= \sum_{\ell\geq 1}\sum_{\substack{ a,b\in [ 4\ell+2]\\ \pr{2\ell+1}\leq a-b\leq  \pr{4\ell+2}-1   }   } \sum_{v=1}^{\ent{\frac{n}{4\ell+2}}-1} 
\pr{ H^{\pr{\ell}}_{ a,\pr{v+1}\pr{4\ell+2}+b}  
+H^{\pr{\ell}}_{ \pr{4\ell+2}+b,v\pr{4\ell+2}+a}}
\end{equation}
 \begin{equation}
R_{n,5}= \sum_{\ell\geq 1}\sum_{\substack{ a,b\in [ 4\ell+2]\\ \pr{2\ell+1}\leq a-b\leq  \pr{4\ell+2}-1   }   } \sum_{ u=0}^{ \ent{\frac{n}{4\ell+2}}-2} 
\pr{H^{\pr{\ell}}_{u\pr{4\ell+2}+a,\pr{u+1}\pr{4\ell+2}+b}-H^{\pr{\ell}}_{u\pr{4\ell+2}+a,2m\pr{2\ell+1}+b}}
\end{equation}
 \begin{equation}
R_{n,6}=  \sum_{\ell\geq 1}\sum_{\substack{ a,b\in [ 4\ell+2]
\\ \pr{2\ell+1}\leq a-b\leq  \pr{4\ell+2}-1   }   } 
\sum_{v=1}^{\ent{\frac{n}{4\ell+2}}-1} 
\pr{ H^{\pr{\ell}}_{ b,v\pr{4\ell+2}+a}- H^{\pr{\ell}}_{v\pr{4\ell+2}+b,v\pr{4\ell+2}+a}},
\end{equation} 
with 
\begin{multline}
H^{\pr{\ell}}_{i,j}:= 
h\pr{f_\ell\pr{V_{i,\ell}},f_\ell\pr{V_{j,\ell}}}
-\E{h\pr{f_\ell\pr{V_{i,\ell}},f_\ell\pr{V_{j,\ell}}}}\\-
\pr{h\pr{f_{\ell-1}\pr{V_{i,\ell-1}},f_{\ell-1}\pr{V_{j,\ell-1}}}
-\E{h\pr{f_{\ell-1}\pr{V_{i,\ell-1}},f_{\ell-1}\pr{V_{j,\ell-1}}}}}.
\end{multline}

\end{Proposition}
 
 \subsection{Law of large numbers} \label{subsec:LLN}

We first present a result on the Marcinkievicz law of large numbers for 
U-statistics of i.i.d. data. In order to extend it to the context of 
functional of i.i.d., we need a control on a maximal function.
 
 \begin{Proposition}[Marcinkievicz-Zygmund law of large numbers]
 \label{prop:MZSSLNiid}
Let $\pr{\eps_i}_{i\in \Z}$ be an i.i.d. sequence with values in $\mathbb R^d$ and let 
$h\colon \R^d\times\R^d\to \R$ be a measurable function. For $1\leq p<2$, define
\begin{equation}
M_{p}:=\sup_{n\geq 1}\frac{1}{n^{2/p}}\abs{U_n^{\ind}\pr{h,\pr{\eps_i}_{i\in \Z}} }.
\end{equation}
Assume that $\E{\abs{h\pr{\eps_0,\eps_1}}^p}$ is finite and that 
for all $x\in\mathbb R^d$, $\E{h\pr{\eps_0,x}}=0$. Then the following
statements hold:
\begin{enumerate}
\item\label{itm:law_of_large_numbers_iid} the sequence $\pr{n^{-2/p}
\abs{U_n^{\ind}\pr{h,\pr{\eps_i}_{i\in \Z}} 
}  }_{n\geq 1}$ converges
 to zero almost surely ;
\item\label{itm:contrl_Mp_iid} for any positive $x$, 
\begin{equation}\label{eq:contrl_Mp_iid_cor}
x^p\PP\ens{M_p >x }\leq \kappa_p 
\E{\abs{h\pr{\eps_0,\eps_1}}^p},
\end{equation}
where $\kappa_p$ is bigger than $1$ and depends only on $p$.
\end{enumerate}
\end{Proposition} 

\begin{Corollary}\label{cor:MZSSLNiid}
 Let $\pr{\eps_i}_{i\in \Z}$ be an i.i.d. sequence with values in $\mathbb R^d$ and let 
$h\colon \R^d\times\R^d\to \R$ be a measurable function. For $1\leq p<2$, define
\begin{equation}
M_{p}:=\sup_{n\geq 1}\frac{1}{n^{1+1/p}}\abs{U_n^{\ind}\pr{h,\pr{\eps_i}_{i\in \Z}} }.
\end{equation}
Assume that $\E{\abs{h\pr{\eps_0,\eps_1}}^p}$ is finite and
$\E{h\pr{\eps_0,\eps_1}}=0$. Then the following
statements hold:
\begin{enumerate}
\item\label{itm:law_of_large_numbers_iid_cor} the sequence $\pr{n^{-1-1/p}
\abs{U_n^{\ind}\pr{h,\pr{\eps_i}_{i\in \Z}} 
}  }_{n\geq 1}$ converges
 to zero almost surely ;
\item\label{itm:contrl_Mp_iid_cor} for any positive $x$, 
\begin{equation}\label{eq:contrl_Mp_iid}
x^p\PP\ens{M_p >x }\leq \kappa_p 
\E{\abs{h\pr{\eps_0,\eps_1}}^p},
\end{equation}
where $\kappa_p$ is bigger than $1$ and depends only on $p$.
\end{enumerate}

\end{Corollary}

The previous Proposition gives a control in terms of the 
weak-$\mathbb L^p$-semi-norm, defined as $\norm{X}_{p,w}:= 
\sup_{x>0}\pr{x^p\PP\ens{\abs X>x}}^{1/p}$, $p>1$. However, the 
triangular inequality may fail, which is not convenient in view of the 
decomposition obtained in Proposition~\ref{prop:Hoeffding_decomposition}. 
For this reason, we introduce the weak-$\mathbb L^p$-norm, defined as 
\begin{equation}
 \norm{X}_{p,\infty}:= \sup_{A:\PP\pr{A}>0}\PP\pr{A}^{1/p-1}
 \E{\abs{X}\mathbf 1_A}.
\end{equation}
Notice that for all fixed $p>1$, there exists constants $c_p$ and $C_p$ 
such that for all random variable $X$, $c_p\norm{X}_{p,w}\leq 
\norm{X}_{p,\infty}\leq C_p\norm{X}_{p,w}$. Moreover, the observation that for any sequence 
of non-negative random variables $\pr{Y_n}_{n\geq 1}$, 
\begin{equation}\label{eq:norm_lp_faible_supremum}
\norm{\sup_{n\geq 1}Y_n}_{p,w}^p\leq \sum_{n\geq 1}
\norm{ Y_n}_{p,w}^p
\end{equation}
will be useful in the sequel.

We are now in position to present our first result, which gives a control 
on the maximal function of a U-statistic whose data is a functional 
of an i.i.d. sequence.

 \begin{Theorem}
 \label{thm:bounded_LLN}
 Let $\pr{\eps_u}_{u\in\Z}$ be an i.i.d. sequence, 
 $f\colon\R^\Z\to \R$ and $h\colon\R\times \R\to\R$  be   measurable functions. Let $p\in [1,2)$. Then 
 \begin{multline}
  \norm{\sup_{n\geq 1}\frac 1{n^{1+1/p}}\abs{
  U_n\pr{h,f,\pr{\eps_u}_{u\in\Z}-
  \E{U_n\pr{h,f,\pr{\eps_u}_{u\in\Z}
  }}
  } }}_{p,\infty}\\
  \leq c_p \pr{\theta_{0,p}+
  \sum_{\ell\geq 1} \ell^{1-1/p}\theta_{\ell,p}}.
 \end{multline}
 \end{Theorem}

 The second result of this Subsection is a Marcinkievicz law of large numbers.
 
 \begin{Theorem}[Marcinkievicz law of large numbers]
 \label{thm:MarcinK_LLN}
 Let $\pr{\eps_u}_{u\in\Z}$ be an i.i.d. sequence, 
 $f\colon\R^\Z\to \R$ and $h\colon\R\times \R\to\R$  be   measurable functions. Let $p\in [1,2)$. Suppose that 
 \begin{equation}
  \sum_{\ell\geq 0}\ell^{1-1/p}\theta_{\ell,p}<+\infty.
 \end{equation}
 Then the following almost sure convergence holds:
 \begin{equation}
  \frac 1{n^{1+1/p}}\abs{
  U_n\pr{h,f,\pr{\eps_u}_{u\in\Z}-\E{U_n\pr{h,f,\pr{\eps_u}_{u\in\Z}
  }}
  } }\to 0.
 \end{equation}

 \end{Theorem}
 
 A law of large numbers has been obtained in 
 \cite{MR1851171}, Theorem~6. They considered the classical 
 law of large numbers ($p=1$) but for functions of a 
 $\beta$-mixing process. When applied to a Lipschitz-continuous 
 kernel and a function of an i.i.d. process, their condition 
 requires 
 \begin{equation}
\sum_{\ell\geq 1}   \norm{f\pr{ \pr{\eps_{-u}}_{u\in\Z}}- 
   \E{f\pr{ \pr{ \eps_{-u}}_{u\in\Z}}\mid \eps_{-\ell}, 
   \dots,\eps_\ell}  }_1<+\infty.
 \end{equation}
When $ f\pr{ \pr{ \eps_{-u}}_{u\in\Z}}$ is a linear process 
with Bernoulli innovations, that is, 
$f\pr{ \pr{\eps_{-u}}_{u\in\Z}}=\sum_{i\in\Z}a_i\eps_{-i}$ 
and $\PP\ens{\eps_0=1}=\PP\ens{\eps_0=-1}=1/2$, 
this is equivalent, by Marcinkievicz-Zygmund inequality, to 
\begin{equation}
\sum_{\ell\geq 1}
\pr{\sum_{i\in\Z,\abs{i}\geq\ell}a_i^2  }^{1/2} <+\infty,
 \end{equation}
while our condition read $\sum_{i\in\Z}\abs{a_i}<+\infty$ and 
is less restrictive. However, in \cite{MR1851171}, the case 
of functional of $\beta$-mixing process and of a not necessarily 
continuous kernel is treated there and not in the present paper.
 
 \subsection{Bounded law of the iterated logarithms} 
 \label{subsec:LLI}
 
Let us present our next result concerning the bounded law 
of the iterated logarithms. Like in the independent case, 
one can control the moments of order $p\in \pr{1,2}$ of 
the maximal function of U-statistic with the normalisation 
$n^{3/2}\LL{n}$.

 \begin{Theorem}\label{thm:LLI_Bernoulli}
   Let $\pr{\eps_u}_{u\in\Z}$ be an i.i.d. sequence, 
 $f\colon\R^\Z\to \R$ and $h\colon\R\to\R$  be   measurable functions. Let $p\in [1,2)$. For all $1\leq p<2$, the following inequality holds:
 \begin{multline}
  \norm{\sup_{n\geq 1}
  \frac 1{n^{3/2}
  \sqrt{\LL{n}}}\abs{U_n\pr{h,f,\pr{\eps_u}_{u\in\Z}
 -\E{U_n\pr{h,f,\pr{\eps_u}_{u\in\Z}
  }} }}}_p \\ \leq c_p  
  \pr{\theta_{0,2}+\sum_{\ell\geq 1}\ell^{1/2}\theta_{\ell,2}},
 \end{multline}
 where $c_p$ depends only on $p$.
 \end{Theorem}

 \subsection{Central limit theorem} 
  \label{subsec:TLC}
 
Let us present our result concerning the central limit theorem. We will 
make essentially three assumptions on the dependence of our U-statistic. 
The first and second one involve the coefficients $\delta_{\cdot,2}$ and 
$\theta_{\ell,1}$ respectively and the third condition is imposed 
is summability of the family of covariances of a strictly stationary 
sequence which comes from the generalized Hoeffding decomposition 
of Proposition~\ref{prop:Hoeffding_decomposition}.
 
\begin{Theorem}\label{thm:TLC}
Let $h\colon\R^2\to\R$ be a measurable function, $f\colon \R^{\Z}\to 
\R$ be a measurable function and $\pr{\eps_u}_{u\in \Z}$ be an 
i.i.d. sequence. Let 
\begin{equation}
X_j:=f\pr{
\pr{\eps_{j-k}}_{k\in \Z}
}, \quad U_n:=\sum_{1\leq i<j\leq n}h\pr{X_i,X_j}.
\end{equation}
Suppose that 

 \begin{equation}\label{eq:hypothese_sur_theta_TLC}
  \sum_{\ell\geq 0}\ell^{1/2}\theta_{\ell,2}<+\infty; \quad
  \sum_{\ell\geq 0}\ell^{2}\theta_{\ell,1}<+\infty
 \end{equation}
 and that 
 \begin{equation}\label{eq:hypothese_covariance}
 \sum_{k\in\Z}\abs{ 
\operatorname{Cov}\pr{Y_0,Y_k} 
 }<+\infty,
 \end{equation}
where the random variable $Y_k$ is defined by the following 
$\mathbb L^2$-convergence 
\begin{equation}
Y_k=\lim_{\ell\to +\infty}
 \E{h\pr{f_\ell\pr{V_{k,\ell}},f_{\ell}\pr{  V'_{0,\ell}  }} \mid V_{k,\ell}  }
-\E{h\pr{f_\ell\pr{V_{k,\ell}},f_{\ell}\pr{  V'_{0,\ell}  }}   }
\end{equation}
and $V_{k,\ell}=\pr{\eps_u}_{u=k-\ell}^{k+\ell}$, 
$V'_{k,\ell}=\pr{\eps'_u}_{u=k-\ell}^{k+\ell}$, where 
$\pr{\eps'_u}_{u\in \Z}$ is an independent copy of $\pr{\eps_u}_{u\in \Z}$.

Then the following convergence in distribution holds:
\begin{equation}
\frac 1{n^{3/2}}\pr{U_n-\E{U_n}}\to N\pr{0,\sigma^2},
\end{equation}
where 
\begin{equation}
\sigma^2:=\sum_{k\in\Z} 
\operatorname{Cov}\pr{Y_0,Y_k}.
\end{equation}
\end{Theorem} 
 
Let us give some comments about this results and how it 
relates with existing results.
 
\begin{Remark}
It is not clear to us whether the condition \eqref{eq:hypothese_sur_theta_TLC} 
guarantees \eqref{eq:hypothese_covariance}. Nevertheless, conditions of the 
type \begin{equation}
\sum_{j\geq 0}\norm{\E{Y_j\mid \sigma\pr{\eps_u,u\leq 0}} 
-\E{Y_j\mid \sigma\pr{\eps_u,u\leq -1}} }_2<\infty
\end{equation}
guarantee \eqref{eq:hypothese_covariance} and are not too hard to check for the 
choices of kernel involved in Subsection~\ref{subsec:applications}.
\end{Remark}
\begin{Remark}
In the context of Theorem~\ref{thm:TLC}, it is possible the 
variance $\sigma^2$ equals $0$. In this case, the normalization 
should change and a deeper analysis is required in order to 
identify the limiting process.
\end{Remark}
\begin{Remark}
It does not seem to be an easy task to compare the condition of 
Theorem~\ref{thm:TLC} with those in \cite{MR2060311} in full generality. One way to measure the 
dependence  in the latter paper is given (in the case where 
the weights are all equal to one) by the coefficients
 \begin{equation}\label{eq:definition_delta_l}
   \delta_{\ell}:= 
   \sup_{j\geq 0}\norm{h\pr{\pr{\eps_{-u}}_{u\in \Z},\pr{\eps_{j-u}}_{u\in \Z}   }-
   h\pr{f_\ell\pr{V_{0,\ell}},f_\ell\pr{V_{j,\ell}}  }
   }_2.
  \end{equation}
  A sufficient condition (which follows from a combination of Theorem~3 
  and Proposition~4) for the central limit theorem is 
  \begin{equation}\label{eq:cond_TLC_Wu}
  \sum_{i\geq 1}\sup_{\ell\geq i}\delta_{\ell}<+\infty.
  \end{equation}
  In a particular case, this condition can be compared with ours. 
  Let $h\colon \pr{x,y}\mapsto x+y$  and $\pr{X_j}_{j\geq 1}$ be 
  a causal linear process:
  \begin{equation}
  X_j:=\sum_{i\geq 0}a_i\eps_{j-i},
  \end{equation}
  where $\pr{\eps_u}_{u\in\Z}$ is i.i.d. centered and square integrable and 
  $\pr{a_i}_{i\in\Z}$ is square summable, 
  conditon \eqref{eq:cond_TLC_Wu} reads 
  \begin{equation}
  \sum_{i\geq 1}\sqrt{\sum_{\ell\geq i}a_\ell^2   }<+\infty
  \end{equation}
  (see Theorem~9). With the same kernel and definition of $\pr{X_j}_{j\geq 1}$, 
  our condition is more restrictive, since the convergence of 
  $\sum_{\ell\geq 0}\ell^{2}\theta_{\ell,1}$ implies that of 
  $\sum_{\ell\geq 0}\ell^{2}a_{\ell}$.
  
  Nevertheless, our result deal with not necessarily causal process. 
\end{Remark}
\begin{Remark}
A central limit theorem has been obtained in 
 \cite{MR1851171}, Theorem~7, for functions of a 
 $\beta$-mixing process. When applied to a Lipschitz-continuous 
 kernel and a function of an i.i.d. process, their condition 
 requires 
 \begin{equation}
\sum_{\ell\geq 1} \ell^2  \norm{f\pr{ \pr{\eps_{-u}}_{u\in\Z}}- 
   \E{f\pr{ \pr{ \eps_{-u}}_{u\in\Z}}\mid \eps_{-\ell}, 
   \dots,\eps_\ell}  }_1<+\infty.
 \end{equation}
When $ f\pr{ \pr{ \eps_{-u}}_{u\in\Z}}$ is a linear process 
with Bernoulli innovations, that is, 
$f\pr{ \pr{\eps_{-u}}_{u\in\Z}}=\sum_{i\in\Z}a_i\eps_{-i}$, 
this is equivalent  to 
\begin{equation}
\sum_{\ell\geq 1}\ell^2
\pr{\sum_{i\in\Z,\abs{i}\geq\ell}a_i^2  }^{1/2} <+\infty,
 \end{equation}
while our condition read $\sum_{\ell\geq 1}\ell^2\abs{a_\ell}
+\ell^2a_{-\ell}<+\infty$ and 
is less restrictive. However, in \cite{MR1851171}, the case 
of functional of $\beta$-mixing process and of a not necessarily 
continuous kernel is treated there and not in the present paper.
\end{Remark} 

\begin{Remark}
In \cite{MR2948902}, an approach based on empirical processes 
has been done in order to derive a central limit theorem 
for $U$-statistics of dependent data. As mentioned in 
Example~3.8, the case of function of i.i.d. processes can be 
treated, since the convergence of empirical processes of 
functions of an i.i.d. sequence has been treated in \cite{MR2384990}. However, in the later paper, the condition 
on the weak dependence relies on the physical dependence 
measure of $\mathbf 1\ens{f\pr{\pr{\eps_{i-u}}_{u\in\Z} } \leq 
t}$ and this does not seem to be comparable with our 
assumptions.
\end{Remark}

 \subsection{Applications}\label{subsec:applications}
  
In this Subsection, we give examples of kernels 
$h$ for which the measure of dependence defined by 
\eqref{eq:definition_de_la_mesure_de_dependence} can be estimated. 

\begin{enumerate}
 \item Uniformly continuous kernel
 
 Let $h\colon \R^2\to\R$ be a measurable function. We assume that 
 there exists a non-negative function $\omega\colon \R_+\to \R_+$ which is 
 increasing, satisfies $\omega\pr{0}=0$ and for each $x,x',y,y'\in\R$,
 \begin{equation}
 \abs{ 
h\pr{x,y}-h\pr{x',y'} 
 }\leq \omega\pr{\abs{x-x'}}+\omega\pr{\abs{y-y'}}.
   \end{equation}
Then the following inequality holds: 
\begin{multline}\label{eq:controle_coeff_dep_unif_cont}
 \norm{  h\pr{f_\ell\pr{V_{0,\ell}},f_\ell\pr{V_{j,\ell}}  }
   -h\pr{f_{\ell-1}\pr{V_{0,\ell-1}},f_{\ell-1}\pr{V_{j,\ell-1}}  }}_p\\
 \leq \norm{\omega\pr{\abs{ f_\ell\pr{V_{0,\ell}}-f_{\ell-1}\pr{V_{0,\ell-1}}  }  }}_p
 +\norm{\omega\pr{\abs{ f_\ell\pr{V_{j,\ell}}-f_{\ell-1}\pr{V_{j,\ell-1}}  }  }}_p.
\end{multline}
Since $\pr{V_{j,\ell},V_{j,\ell-1}}$ has the same distribution as 
$\pr{V_{0,\ell},V_{0,\ell-1}}$, the two terms of the right hand side 
of \eqref{eq:controle_coeff_dep_unif_cont} are equal hence 
\begin{equation}\label{eq:controle_general_coeff_dependence}
\theta_{\ell,p}\leq 2\norm{\omega\pr{\abs{ f_\ell\pr{V_{0,\ell}}-f_{\ell-1}\pr{V_{0,\ell-1}}  }  }}_p.
\end{equation}

In particular, if $h$ is $\alpha$-Hölder continuous for some $\alpha\in\pr{0,1}$, we
 can choose 
$\omega\colon t\mapsto ct^\alpha$ for some constant $c$ and in this case, 
the estimate \eqref{eq:controle_general_coeff_dependence} becomes 
\begin{equation}
\theta_{\ell,p}\leq 2c
\norm{ f_\ell\pr{V_{0,\ell}}-f_{\ell-1}\pr{V_{0,\ell-1}}}_{p\alpha}^\alpha,
\end{equation}
which can be rewritten as 
\begin{equation}
\theta_{\ell,p}\leq 2c
\norm{\E{X_0\mid V_{0,\ell}}-\E{X_0\mid V_{0,\ell-1}}}_{p\alpha}^\alpha.
\end{equation}

 \item Variance estimation. Consider the kernel $h\colon\R^2\to\R$  defined 
 by $h\pr{x,y}:=\pr{x-y}^2/2$.  The associated U-statistic is (after normalization) 
 the classical variance estimator. For this choice of kernel, one can estimate 
 the measure of dependence defined
  by \eqref{eq:definition_de_la_mesure_de_dependence}.
 One needs to control the $\mathbb L^p$-norm of the difference of the square of
 two random variables $Y$ and $Z$. Since 
 \begin{equation}
  \norm{Y^2-Z^2}_p^p= 
  \E{\abs{Y-Z}^p\abs{Y+Z}^p}\leq \E{\abs{Y-Z}^{2p}}^{1/2}\E{\abs{Y+Z}^{2p}}^{1/2}
 \end{equation}
 we derive that 
 \begin{equation}
  \norm{Y^2-Z^2}_p  \leq \norm{Y-Z}_{2p}\pr{\norm{Y}_{2p}+
  \norm{Z}_{2p}}.
 \end{equation}
 We use this for a fixed $\ell\geq 1$ to $Y=
 f_\ell\pr{V_{0,\ell}}-f_\ell\pr{V_{j,\ell}}$ and 
 $Z=f_{\ell-1}\pr{V_{0,\ell-1}}-f_{\ell-1}\pr{V_{j,\ell-1}} $.
 Accounting the following bounds (which are a consequence of 
 stationarity):
 \begin{equation}
  \norm{Y}_{2p}+
  \norm{Z}_{2p}\leq 2 \norm{f_\ell\pr{V_{0,\ell}}}_{2p}
  +2 \norm{f_\ell\pr{V_{0,\ell-1}}}_{2p}\leq 4
  \norm{X_0}_{2p},
 \end{equation}
 we get that 
 \begin{equation}
  \theta_{\ell,p}\leq 2\norm{X_0}_{2p} 
  \norm{\E{X_0\mid V_{0,\ell}}-\E{X_0\mid V_{0,\ell-1}}  }_p.
 \end{equation}

\end{enumerate}

\section{Proofs}  \label{sec:proofs}
  
\subsection{Proof of the generalised Hoeffding's decomposition}

Let us explain the idea of proof of  Proposition~\ref{prop:Hoeffding_decomposition}.
First, we write $U_n\pr{h,f,\pr{\eps_i}_{i\in \Z}  }$ as a series (index by $\ell$) of U-statistics. For the term of index $\ell$, the data of the corresponding U-statistic 
is a function of $2\ell+1$ i.i.d. random variables. We then decompose this a sum of 
$4\pr{2\ell+1}$ U-statistics of independent data. 

We divide the proof in two steps:
\begin{enumerate}
\item\label{step1_Hoeffding_decomposition} First we 
treat the case where $X_j=f\pr{\pr{\eps_u  }_{u=j-\ell}^{j+\ell}}$ and 
$h\colon\R \times \R \to \R$ a symmetric function.
\item\label{step2_Hoeffding_decomposition} Then we write the mentioned U-statistic a sum of U-statistics tractable 
with the work of step \ref{step1_Hoeffding_decomposition}, plus a remainder term. 
\end{enumerate}

Step \ref{step1_Hoeffding_decomposition}: we decompose a
U-statistic whose data is a function of $2\ell+1$ i.i.d. random variables as 
a sum of $4\ell+2$ U-statistics of i.i.d. data plus remainder terms. 

\begin{Lemma}\label{lem:Hoeffding_decomposition_ell_dependent}
Let $\ell \geq 1$ be an  integer,  $h\colon\R^{2\ell+1}
\times\R^{2\ell+1}\to \R$
 be a  measurable function, let $\pr{\eps_u}_{u\in\Z}$ be an i.i.d. real-valued sequence
 and $\pr{\eps'_u}_{u\in\Z}$ an independent copy of $\pr{\eps_u}_{u\in\Z}$. Define 
$U_n\pr{h, \pr{\eps_u}_{u\in\Z}}:=\sum_{1\leq i<j\leq n}  h\pr{ V_i , 
V_j}=\sum_{1\leq i<j\leq n}H_{i,j}$, where 
$V_j= \pr{\varepsilon_u}_{u=j-\ell}^{j+\ell} $ 
and $H_{i,j}:=   h\pr{ V_i,V_j}$.

Then 
\begin{equation}
U_n \pr{h, \pr{\eps_u}_{u\in\Z}}=
\sum_{a,b\in [ 4\ell+2]} U_{\ent{\frac{n}{ 4\ell+2}}}^{\ind}\pr{h_{a,b}   ,
\pr{\eps_i^{a,b}}   }+\sum_{k=1}^6R_{n,k},
\end{equation}
where
\begin{itemize}
\item the function  $h_{a,b}\colon \R^{4\ell+2}\times\R^{4\ell+2}\to \R$ is defined for 
$0\leq a-b\leq  2\ell$ by 
\begin{equation}\label{eq:definition_hab_ab_a_moins_b_petit}
h_{a,b}\pr{\pr{x_i}_{i=1}^{4\ell+2}, \pr{y_j}_{j=1}^{4\ell+2}  } 
:=h \pr{
 \pr{x_{i+a-b}}_{i=1}^{2\ell+1}   ,
 \pr{y_{j} }_{j=1}^{2\ell+1}  }
+ h \pr{
 \pr{y_{j+a-b}}_{j=1}^{2\ell+1}   ,
 \pr{x_{i} }_{i=1}^{2\ell+1}  } 
\end{equation}
and for $\pr{2\ell+1}\leq a-b\leq  \pr{4\ell+2}-1$ by 
\begin{multline}\label{eq:definition_hab_ab_a_moins_b_grand}
h_{a,b}\pr{\pr{x_i}_{i=1}^{4\ell+2}, \pr{y_j}_{j=1}^{4\ell+2}  }  \\
:=h\pr{\pr{x_{i+a-b-\pr{2\ell+1}}}_{i=1}^{2\ell+1}  ,
 \pr{y_{j} }_{j=1}^{2\ell+1}  } + h\pr{
 \pr{y_{j+a-b-\pr{2\ell+1}}}_{j=1}^{2\ell+1}  ,
 \pr{x_{i} }_{i=1}^{2\ell+1}  },
\end{multline}
\item the random vectors $\varepsilon_u^{a,b}$ are defined by 
\begin{equation}\label{eq:definition_eps_ab_a_moins_b_petit}
\varepsilon^{a,b}_u:=\pr{\eps_j}_{j=u\pr{4\ell+2}+b-\ell}^{\pr{u+1}\pr{4\ell+2}+b-\ell-1 }, 
0\leq a-b\leq  2\ell;
\end{equation}
\begin{equation}\label{eq:definition_eps_ab_a_moins_b_grand}
\varepsilon^{a,b}_u:=\pr{\eps_j}_{j=u\pr{4\ell+2}+b+\ell +1}^{\pr{u+1}\pr{4\ell+2}+b  +\ell}, 
\pr{2\ell+1} \leq a-b\leq  \pr{4\ell+2}-1;
\end{equation}
\item the remainder terms are defined (with the convention that $\sum_{
u=1}^{-k}=0$, $k\leq 0$)  by
\begin{equation}\label{eq:definition_de_Rn1}
R_{n,1}=\sum_{1\leq i<j\leq n}
H_{i,j}-\sum_{1\leq i<j\leq  \pr{4\ell+2}\ent{\frac{n}{4\ell+2}}}
H_{i,j};
\end{equation}
\begin{equation}\label{eq:definition_de_Rn2}
R_{n,2}:=\sum_{u=0}^{\ent{\frac{n}{4\ell+2}}-1}
\sum_{\substack{ a,b\in [ 4\ell+2]\\ a<b  }   }H_{u\pr{4\ell+2}+a,u\pr{4\ell+2}+b};
\end{equation}
 \begin{equation}\label{eq:definition_de_Rn3}
R_{n,3}= \sum_{v=1}^{\ent{\frac{n}{4\ell+2}}-1}
\sum_{\substack{ a,b\in [ 4\ell+2]\\ 0\leq a-b\leq \pr{2\ell+1}-1   }   }
\pr{H_{ a,v\pr{4\ell+2}+b}+H_{ b,v\pr{4\ell+2}+a}};
\end{equation}
 \begin{equation}\label{eq:definition_de_Rn4}
R_{n,4}= \sum_{\substack{ a,b\in [ 4\ell+2]\\ \pr{2\ell+1}\leq a-b\leq  \pr{4\ell+2}-1   }   } \sum_{v=1}^{\ent{\frac{n}{4\ell+2}}-1} 
\pr{ H_{ a,\pr{v+1}\pr{4\ell+2}+b}  
+H_{ \pr{4\ell+2}+b,v\pr{4\ell+2}+a}};
\end{equation}
 \begin{equation}\label{eq:definition_de_Rn5}
R_{n,5}= \sum_{\substack{ a,b\in [ 4\ell+2]\\ \pr{2\ell+1}\leq a-b\leq  \pr{4\ell+2}-1   }   } \sum_{ u=0}^{ \ent{\frac{n}{4\ell+2}}-2} 
\pr{H_{u\pr{4\ell+2}+a,\pr{u+1}\pr{4\ell+2}+b}-H_{u\pr{4\ell+2}+a,2m\pr{2\ell+1}+b}};
\end{equation}
 \begin{equation}\label{eq:definition_de_Rn6}
R_{n,6}=  \sum_{\substack{ a,b\in [ 4\ell+2]
\\ \pr{2\ell+1}\leq a-b\leq  \pr{4\ell+2}-1   }   } 
\sum_{v=1}^{\ent{\frac{n}{4\ell+2}}-1} 
\pr{ H_{ b,v\pr{4\ell+2}+a}- H_{v\pr{4\ell+2}+b,v\pr{4\ell+2}+a}}.
\end{equation}

\end{itemize}
\end{Lemma}

\begin{proof}
Before going into the details of the proof, let us explain the general idea, which 
will also explain the origin of the remainder terms. First, it turns out that it would 
be more convenient that $n$ is a multiple of $4\ell+2$. However, it has 
no reason to be the case and we should also take into account the difference 
between $U_n$ and $U_{n'}$ where $n'$ is a multiple of $4\ell+2$ close to 
$n$. The control of the difference of these two terms is precisely $R_{n,1}$.
Now that we are reduced to the case where $n$ is a multiple of $4\ell+2$, 
we look that the terms $H_{i,j}$ and the remainder of $i$ and $j$ by the 
Euclidian division by $4\ell+2$. If these remainder are closed enough, 
then we can write the corresponding sum as a U-statistic whose data 
are independent vectors of length $4\ell+2$. If the remainders are too 
far way, we have to add and substract a term to be reduced to the 
previous case, and this leads to the definition of $R_{n,k}$, $3\leq k\leq 6$. 
If the quotient in the Euclidian division by $4\ell+1$ of $i$ and $j$ are the same, then
the corresponding sum is $R_{n,2}$.

From the equality $U_n\pr{h_2,f,\pr{\eps_i}_{i\in\Z}}=\sum_{1\leq i<j\leq n}
H_{i,j}$ we get by \eqref{eq:definition_de_Rn1} that 
\begin{equation}
U_n\pr{h_2,\pr{\eps_i}_{i\in\Z}}
=\sum_{1\leq i<j\leq  \pr{4\ell+2}\ent{\frac{n}{4\ell+2}}}
H_{i,j}+R_{n,1}.
\end{equation}
For simplicity, let us denote by $I$ the set $[ 4\ell+2]$ and $m:= 
\ent{\frac{n}{4\ell+2}}$. With these notations and in view of 
\eqref{eq:definition_de_Rn2}, the following equality takes place 
\begin{multline}\label{eq:demo_dec_Hoeffding_ell_dep_etape 1}
\sum_{1\leq i<j\leq  \pr{4\ell+2}\ent{\frac{n}{4\ell+2}}}
H_{i,j}\\ 
=\sum_{0\leq u<v\leq m-1}
\sum_{\substack{ a,b\in I\\ 0\leq a-b\leq \pr{2\ell+1}-1   }   }
\pr{H_{u\pr{4\ell+2}+a,v\pr{4\ell+2}+b}+H_{u\pr{4\ell+2}+b,v\pr{4\ell+2}+a}}\\
+\sum_{0\leq u<v\leq m-1}
\sum_{\substack{ a,b\in I\\ \pr{2\ell+1}\leq a-b\leq  \pr{4\ell+2}-1   }   }
\pr{H_{u\pr{4\ell+2}+a,v\pr{4\ell+2}+b}+H_{u\pr{4\ell+2}+b,v\pr{4\ell+2}+a}}
+R_{n,2}.
\end{multline}
Let us treat the first term. For $a,b\in I$ such that $0\leq a-b\leq 2\ell$, 
in view of the definitions \eqref{eq:definition_hab_ab_a_moins_b_petit} 
and \eqref{eq:definition_eps_ab_a_moins_b_petit}  and the symmetry of $h_2$, 
the following equaliy holds
\begin{equation}
h_{a,b}\pr{\eps_u^{a,b},\eps_v^{a,b}}
=H_{u\pr{4\ell+2}+a,v\pr{4\ell+2}+b}+H_{u\pr{4\ell+2}+b,2v\pr{4\ell+2}+a}
\end{equation}
hence 
\begin{multline}\label{eq:demo_dec_Hoeffding_ell_dep_etape2}
\sum_{0\leq u<v\leq m-1}
\sum_{\substack{ a,b\in I\\ 0\leq a-b\leq \pr{2\ell+1}-1   }   }
\pr{H_{u\pr{4\ell+2}+a,v\pr{4\ell+2}+b}+H_{u\pr{4\ell+2}+b,v\pr{4\ell+2}+a}}\\ 
= 
U_{m}^{\ind}\pr{h_{a,b},\pr{\varepsilon_{u}^{a,b}}_{u\in \Z}   }+R_{n,3}.
\end{multline}
Let us treat the second term in the right hand side of 
\eqref{eq:demo_dec_Hoeffding_ell_dep_etape 1}. 
Adding and stubstracting the terms 
$ H_{u\pr{4\ell+2}+a,\pr{v+1}\pr{4\ell+2}+b}$ and $H_{\pr{u+1}\pr{4\ell+2}+b,v\pr{4\ell+2}+a}$ gives, after having rewriten the corresponding sums 
as double sums and exploited a telescoping of the inside sum, 

\begin{equation}
\sum_{0\leq u<v\leq m-1}
\pr{H_{u\pr{4\ell+2}+a,v\pr{4\ell+2}+b}+H_{u\pr{4\ell+2}+b,v\pr{4\ell+2}+a}}\\
= S+R_{n,5}+R_{n,6} ,
\end{equation}

where 
\begin{equation}
S:=  \sum_{0\leq u<v\leq m-1}
\pr{ H_{u\pr{4\ell+2}+a,\pr{v+1}\pr{4\ell+2}+b}  
+H_{2
\pr{u+1}\pr{2\ell+1}+b,v\pr{4\ell+2}+a}} 
\end{equation}
We express $S$ as a U-statistic of independent data. 
Noticing that $0\leq a-\pr{b+\pr{2\ell+1}}\leq \pr{2\ell+1}-1$, we are in a similar 
situation as in the case $0\leq a-b\leq \pr{2\ell+1}-1$, with $b$ replaced 
by $b+\pr{2\ell+1}$. Therefore, in view of 
\eqref{eq:definition_eps_ab_a_moins_b_grand}, 
\eqref{eq:definition_hab_ab_a_moins_b_grand} and \eqref{eq:definition_de_Rn4},
we obtain that 
\begin{equation}
S=U_{m-1}^{\ind}\pr{ 
h_{a,b},\pr{\eps_u^{a,b}}_{u\in\Z}}+R_{n,4}.
\end{equation}

Collecting these terms gives 
\begin{multline}\label{eq:demo_dec_Hoeffding_ell_dep_etape3}
\sum_{\substack{ a,b\in I\\ \pr{2\ell+1}\leq a-b\leq  \pr{4\ell+2}-1   }   }\sum_{0\leq u<v\leq m-1}
\pr{H_{u\pr{4\ell+2}+a,v\pr{4\ell+2}+b}+H_{u\pr{4\ell+2}+b,v\pr{4\ell+2}+a}}\\
=\sum_{\substack{ a,b\in I\\ \pr{2\ell+1}\leq a-b\leq  \pr{4\ell+2}-1   }   } U_{m-1}^{\ind}\pr{ 
h_{a,b},\pr{\eps_u^{a,b}}_{u\in\Z}}+ R_{n,4}+ R_{n,5}+ R_{n,6}.
\end{multline}
We end the proof of Lemma~\ref{lem:Hoeffding_decomposition_ell_dependent} by combining the equalities \eqref{eq:demo_dec_Hoeffding_ell_dep_etape 1}, 
\eqref{eq:demo_dec_Hoeffding_ell_dep_etape2} and 
\eqref{eq:demo_dec_Hoeffding_ell_dep_etape3}.
\end{proof}

Now, we can apply the Hoeffding decomposition to each U-statistic of i.i.d. data 
and rewrite the remainder term in a more tractable form.
\begin{Lemma}\label{lem:Hoeffding_dec_fct_ell_dep}
Let $\ell \geq 1$ be an  integer,  $h\colon\R^{2\ell+1}
\times\R^{2\ell+1}\to \R$
 be a  measurable function, let $\pr{\eps_u}_{u\in\Z}$ be an i.i.d. real-valued sequence
 and $\pr{\eps'_u}_{u\in\Z}$ an independent copy of $\pr{\eps_u}_{u\in\Z}$. Define 
$U_n\pr{h, \pr{\eps_u}_{u\in\Z}}:=\sum_{1\leq i<j\leq n}  h\pr{ V_i , 
V_j}-\E{h\pr{V_i,V_j}}=\sum_{1\leq i<j\leq n}H_{i,j}$, where 
$V_j= \pr{\varepsilon_u}_{u=j-\ell}^{j+\ell} $,  $V'_j=
 \pr{\varepsilon'_u}_{u=j-\ell}^{j+\ell} $
and $H_{i,j}:=   h\pr{ V_i,V_j}-\E{h\pr{V_i,V_j}}$.

Then 
\begin{multline}
U_n \pr{h, \pr{\eps_u}_{u\in\Z}}=
\pr{4\ell+2}\ent{\frac{n}{ 4\ell+2}}\sum_{k=1}^{  \pr{4\ell+2}\pr{\ent{\frac{n}{ 4\ell+2}}} +1 }
\pr{\E{h\pr{V_k,V'_0}\mid V_k}-\E{h\pr{V_k,V'_0} }}\\
+
\sum_{a,b\in [ 4\ell+2]} U_{\ent{\frac{n}{ 4\ell+2}}}^{\ind}\pr{h_{a,b}^{\pr{2}}   ,
\pr{\eps_i^{a,b}}   }+\sum_{k=1}^6R_{n,k},
\end{multline}
where the function  $h_{a,b}\colon \R^{  4\ell+2}\times\R^{  4\ell+2}\to \R$ is defined
\begin{multline}
h_{a,b}^{\pr{2}}
:= h_{a,b}\pr{\pr{x_i}_{i=1}^{4\ell+2}, \pr{y_j}_{j=1}^{4\ell+2}  }
\\ - h_{a,b} 
 \pr{ \pr{x_i}_{i=1}^{4\ell+2}, \eps_0^{a,b}  } 
 -h_{a,b}
 \pr{ \pr{y_j}_{j=1}^{4\ell+2} , \eps_0^{a,b}  } 
 +\E{h_{a,b}\pr{\eps_0^{a,b},{\eps_0'}^{a,b}  }}
\end{multline}
and 
\begin{itemize}
\item the function  $h_{a,b}\colon \R^{4\ell+2}\times\R^{4\ell+2}\to \R$ is defined for 
$0\leq a-b\leq  2\ell$ by

\begin{equation}\label{eq:definition_hab_ab_a_moins_b_petit_Hoeffding}
h_{a,b}\pr{\pr{x_i}_{i=1}^{4\ell+2}, \pr{y_j}_{j=1}^{4\ell+2}  }  
:=h \pr{
 \pr{x_{i+a-b}}_{i=1}^{2\ell+1}   ,
 \pr{y_{j} }_{j=1}^{2\ell+1}  }
+ h \pr{
 \pr{y_{j+a-b}}_{j=1}^{2\ell+1}   ,
 \pr{x_{i} }_{i=1}^{2\ell+1}  }
\end{equation}
and for $\pr{2\ell+1}\leq a-b\leq  \pr{4\ell+2}-1$ by 
\begin{multline}\label{eq:definition_hab_ab_a_moins_b_grand_Hoeffding}
h_{a,b}\pr{\pr{x_i}_{i=1}^{4\ell+2}, \pr{y_j}_{j=1}^{4\ell+2}  }  \\
:=h\pr{\pr{x_{i+a-b-\pr{2\ell+1}}}_{i=1}^{2\ell+1}  ,
 \pr{y_{j} }_{j=1}^{2\ell+1}  } + h\pr{
 \pr{y_{j+a-b-\pr{2\ell+1}}}_{j=1}^{2\ell+1}  ,
 \pr{x_{i} }_{i=1}^{2\ell+1}  },
\end{multline}
\item the random vectors $\varepsilon_u^{a,b}$ are defined by 
\begin{equation} 
\varepsilon^{a,b}_u:=\pr{\eps_j}_{j=u\pr{4\ell+2}+b-\ell}^{\pr{u+1}\pr{4\ell+2}+b-\ell-1 }, 
0\leq a-b\leq  2\ell;
\end{equation}
\begin{equation} 
\varepsilon^{a,b}_u:=\pr{\eps_j}_{j=u\pr{4\ell+2}+b+\ell +1}^{\pr{u+1}\pr{4\ell+2}+b  +\ell}, 
\pr{2\ell+1} \leq a-b\leq  \pr{4\ell+2}-1;
\end{equation}
\item the remainder terms are defined (with the convention that $\sum_{
u=1}^{-k}=0$, $k\leq 0$)  by
\begin{equation}
R_{n,1,1}:= \sum_{j=\pr{4\ell+2}\ent{\frac{n}{4\ell+2}   }+1     }^n 
\sum_{i=\pr{4\ell+2}\ent{\frac{j-1}{4\ell+2}   }+1 }^{j-1}H_{i,j}
\end{equation}
\begin{equation}
R_{n,1,2}:= \sum_{a=1}^{4\ell+2}\sum_{j=\pr{4\ell+2}\ent{\frac{n}{4\ell+2}   }+1     }^n 
\sum_{k=0}^{\ent{\frac{j-1}{4\ell+2}   }-1}H_{\pr{4\ell+2}k+a,j};
\end{equation}
\begin{equation}
R_{n,2}:=\sum_{u=0}^{\ent{\frac{n}{4\ell+2}}-1}
\sum_{\substack{ a,b\in [ 4\ell+2]\\ a<b  }   }H_{u\pr{4\ell+2}+a,u\pr{4\ell+2}+b}
\end{equation}
 \begin{equation}
R_{n,3}:= \sum_{v=1}^{\ent{\frac{n}{4\ell+2}}-1}
\sum_{\substack{ a,b\in [ 4\ell+2]\\ 0\leq a-b\leq \pr{2\ell+1}-1   }   }
\pr{H_{ a,v\pr{4\ell+2}+b}+H_{ b,v\pr{4\ell+2}+a}}
\end{equation}
 \begin{equation}
R_{n,4}= \sum_{\substack{ a,b\in [ 4\ell+2]\\ \pr{2\ell+1}\leq a-b\leq  \pr{4\ell+2}-1   }   } \sum_{v=1}^{\ent{\frac{n}{4\ell+2}}-1} 
\pr{ H_{ a,\pr{v+1}\pr{4\ell+2}+b}  
+H_{ \pr{4\ell+2}+b,v\pr{4\ell+2}+a}}
\end{equation}
 \begin{equation}
R_{n,5}= \sum_{\substack{ a,b\in [ 4\ell+2]\\ \pr{2\ell+1}\leq a-b\leq  \pr{4\ell+2}-1   }   } \sum_{ u=0}^{ \ent{\frac{n}{4\ell+2}}-2} 
\pr{H_{u\pr{4\ell+2}+a,\pr{u+1}\pr{4\ell+2}+b}-H_{u\pr{4\ell+2}+a,2m\pr{2\ell+1}+b}}
\end{equation}
 \begin{equation}
R_{n,6}=  \sum_{\substack{ a,b\in [ 4\ell+2]
\\ \pr{2\ell+1}\leq a-b\leq  \pr{4\ell+2}-1   }   } 
\sum_{v=1}^{\ent{\frac{n}{4\ell+2}}-1} 
\pr{ H_{ b,v\pr{4\ell+2}+a}- H_{v\pr{4\ell+2}+b,v\pr{4\ell+2}+a}},
\end{equation}

\end{itemize}
\end{Lemma}
To sum up, the centered U-statistic whose data is a function of $2\ell+1$ i.i.d. 
random variables can be decomposed as a partial sum of a strictly stationary sequence, 
a sum of degenerated U-statistics plus a remainder term.

\begin{proof}[Proof of Proposition~\ref{prop:Hoeffding_decomposition}]
By the assumption \eqref{eq:convergence_presque_sure_pour_decomposition},
the following equality holds almost surely:
\begin{equation}
U_n\pr{h,f,\pr{\eps_u}_{u\in\Z}}
=U_n\pr{h,f_0,\pr{\eps_u}_{u\in\Z}}+\sum_{\ell=1}^{+\infty}\pr{U_n
\pr{h,f_\ell,\pr{\eps_u}_{u\in\Z}}
-U_n
\pr{h,f_{\ell-1},\pr{\eps_u}_{u\in\Z}}}.
\end{equation}
Since $U_n\pr{h,f_0,\pr{\eps_u}_{u\in\Z}}$ is a U-statistic of 
independent identically distributed
 data, it can be treated by the classical Hoeffding's decomposition
written in \eqref{eq:Hoeffding_iid}.

To proceed, we need to decompose for a fixed $\ell\geq 1$ the 
term $U_n
\pr{h,f_\ell,\pr{\eps_u}_{u\in\Z}}
-U_n
\pr{h,f_{\ell-1},\pr{\eps_u}_{u\in\Z}}$. To this aim, we
 apply Lemma~\ref{lem:Hoeffding_dec_fct_ell_dep}
in the following setting:
 the function $h$ is replaced by $\widetilde{h}\colon 
\R^{2\ell+1}\times\R^{2\ell+1}\to \R $, which is defined by
\begin{multline}
\widetilde{h}\pr{ 
\pr{x_i}_{i=1}^{2\ell+1},\pr{y_j}_{j=1}^{2\ell+1}
}\\
=h\pr{f_\ell\pr{\pr{x_i}_{i=1}^{2\ell+1}   }  ,    
f_\ell\pr{\pr{y_j}_{j=1}^{2\ell+1}   }
 }-
 h\pr{f_{\ell-1}\pr{\pr{x_i}_{i=2}^{2\ell}   }  ,    
f_{\ell-1}\pr{\pr{y_j}_{j=2}^{2\ell}   }
 }.
\end{multline}
 
In this way, 
\begin{equation}
U_n
\pr{h,f_\ell,\pr{\eps_u}_{u\in\Z}}
-U_n
\pr{h,f_{\ell-1},\pr{\eps_u}_{u\in\Z}}= 
\sum_{1\leq i<j\leq n}\widetilde{h}
\pr{\pr{\eps_{i-u}}_{u=-\ell}^\ell,
\pr{\eps_{j-v}}_{v=-\ell}^\ell   } 
\end{equation}
 
and with the notation $V_{i,\ell}=\pr{\eps_{u}}_{u=i-\ell}^{i+\ell}$, 
   Lemma~\ref{lem:Hoeffding_decomposition_ell_dependent} gives 
\begin{multline}\label{eq:utilisation_lemme_Hoeffding}
U_n
\pr{h,f_\ell,\pr{\eps_u}_{u\in\Z}}
-U_n
\pr{h,f_{\ell-1},\pr{\eps_u}_{u\in\Z}}
-\E{U_n
\pr{h,f_\ell,\pr{\eps_u}_{u\in\Z}}
-U_n
\pr{h,f_{\ell-1},\pr{\eps_u}_{u\in\Z}}}\\ 
=
\pr{4\ell+2}\ent{\frac{n}{4\ell+2}} \sum_{k=1}^{\pr{4\ell+2}\ent{\frac{n}{4\ell+2}}+1 }\pr{\E{\til{h}\pr{  V_{k,\ell},V'_{0,\ell} } 
\mid V_{k,\ell}}-\E{\til{h}\pr{  V_{k,\ell},V'_{0,\ell} } 
\mid V_{k,\ell}}}
\\
+\sum_{a,b\in [4\ell+2]} U_{\ent{\frac{n}{4\ell+2}}   }^{\ind}\pr{
\til{h}^{\pr{2}}_{a,b}   ,
\pr{ \eps_i^{a,b}}   }+\sum_{k=1}^6R_{n,k},
\end{multline}

Observe that 
\begin{equation}
\E{\til{h}\pr{  V_{k,\ell},V'_{0,\ell} } 
\mid V_{k,\ell}}= \E{h\pr{f_\ell\pr{V_{k,\ell}},f_{\ell}\pr{  V'_{0,\ell}  }} \mid V_{k,\ell}  }
-\E{h\pr{f_{\ell-1}\pr{V_{k,\ell-1}},f_{\ell-1}\pr{  V'_{0,\ell-1}  } \mid V_{k,\ell}  }}
\end{equation}
and using Lemma~\ref{lem:propriete_esperance_cond_2_tribus}, 
this equality becomes 
\begin{equation}
\E{\til{h}\pr{  V_{k,\ell},V'_{0,\ell} } 
\mid V_{k,\ell}}= \E{h\pr{f_\ell\pr{V_{k,\ell}},f_{\ell}\pr{  V'_{0,\ell}  }} \mid V_{k,\ell}  }
-\E{h\pr{f_{\ell-1}\pr{V_{k,\ell-1}},f_{\ell-1}\pr{  V'_{0,\ell-1}  }} \mid V_{k,\ell-1}  }
\end{equation}

Moreover, 
\begin{equation}
U_{\ent{\frac{n}{4\ell+2}}   }^{\ind}\pr{
\til{h}^{\pr{2}}_{a,b}   ,
\pr{ \eps_i^{a,b}}   }
= \sum_{1\leq u<v\leq \ent{\frac{n}{4\ell+2}}}\pr{
Y_{\ell,u,v}^{a,b}+Y_{\ell,u,v}^{b,a} -Y_{\ell-1,u,v}^{a,b} -Y_{\ell-1,u,v}^{b,a} }
\end{equation}

where  
\begin{multline}
Y_{\ell,u,v}^{a,b}= h\pr{f_{\ell}\pr{V_{\pr{4\ell+2}u+a,\ell}}, f_{\ell}\pr{V_{\pr{4\ell+2}v+b,\ell}}  } - \E{h\pr{f_{\ell}\pr{V_{\pr{4\ell+2}u+a,\ell}}, f_{\ell}\pr{V'_{0,\ell}}  }\mid V_{\pr{4\ell+2}u+a,\ell}}\\-
\E{h\pr{f_{\ell}\pr{V_{\pr{4\ell+2}v+b,\ell}}, f_{\ell}\pr{V'_{0,\ell}}  }\mid V_{\pr{4\ell+2}v+b,\ell}}+\E{h\pr{f_\ell\pr{V_{0,\ell}},f_\ell\pr{V'_{0,\ell}}}}.
\end{multline}

We conclude by collecting all the terms.
\end{proof}  
  
\subsection{Proof of the results of Subsection~\ref{subsec:LLN}}

\begin{proof}[Proof of Proposition~\ref{prop:MZSSLNiid}]
It will be more convenient to work with dyadics, since the 
martingale property will be useful to handle the maximums.
First observe that $M_p\leq 2^{1/p} M'_p$, where 
\begin{equation}
M'_p:=\sup_{N\geq 1} 2^{-2N/p}\max_{2\leq n\leq 2^N} 
\abs{U_n^{\ind}\pr{h,\pr{\eps_i}_{i\in \Z}}-\E{U_n^{\ind}\pr{h,\pr{\eps_i}_{i\in \Z}}}
}.
\end{equation}

For a fixed 
integer $n$, consider the event 
\begin{equation}
A_N:=\ens{2^{-2N/p}\max_{2\leq n\leq 2^N}
\abs{U_n^{\ind}\pr{h,\pr{\eps_i}_{i\in \Z}} 
}>2   }.
\end{equation}

It suffices to prove that there exists a constant $c_p$ (depending only on $p$) such that 
\begin{equation}\label{eq:equation_cle_LGN_cas_iid}
\sum_{N=1}^{+\infty}\PP\pr{A_N}\leq c_p\E{\abs{h\pr{\eps_0,\eps_1}}^p  }.
\end{equation}

Indeed, item~\ref{itm:law_of_large_numbers_iid} follows from an application of 
\eqref{eq:equation_cle_LGN_cas_iid} to $h/\varepsilon$ for a positive $\eps$ and the 
Borel-Cantelli lemma. In order to prove item~\ref{itm:contrl_Mp_iid_cor}, we 
notice that $\PP\ens{M'_p>2}\leq \sum_{N=1}^{+\infty}\PP\pr{A_N}$ and we apply 
\eqref{eq:equation_cle_LGN_cas_iid} to $h/x$ for each positive $x$.
Consequently, we focus on establishing a satisfactory bound for $\PP\pr{A_N}$.

Define for $j\geq 2$ the random variable $D_j:=\sum_{i=1}^{j-1}h\pr{\eps_i,\eps_j}$. 
Let $\mathcal F_j$ denote the $\sigma$-algebra generated by the random variables 
$\eps_k,1\leq k\leq j$. Define 
\begin{equation}
D'_j:= D_j\mathbf 1\ens{\abs{D_j}\leq 2^{2N/p}}- 
\E{D_j\mathbf 1\ens{\abs{D_j}\leq 2^{2N/p}}\mid \mathcal F_{j-1}}\mbox{ and }
\end{equation}
\begin{equation}
D''_j:= D_j\mathbf 1\ens{\abs{D_j}> 2^{2N/p}}- 
\E{D_j\mathbf 1\ens{\abs{D_j}> 2^{2N/p}}\mid \mathcal F_{j-1}}.
\end{equation}
Since $\E{D_j\mid\f_{j-1}}=0$, it follows that $D_j=D'_j+D''_j$ hence  $A_N\subset A'_N\cup A''_N$, where 
\begin{equation}
A'_N:=\ens{2^{-4N/p}\max_{2\leq k\leq 2^N}
\abs{\sum_{j=2}^k D'_j}> 1   }\mbox{ and }
\end{equation}
\begin{equation}
A''_N:=\ens{2^{-4n/p}\max_{2\leq k\leq 2^N}\abs{\sum_{j=2}^k D''_j}> 1   }. 
\end{equation}
Let us bound $p'_N:=\PP\pr{A'_N}$. Markov's inequality entails 
\begin{equation}
p'_N\leq 2^{-4N/p}\E{\max_{2\leq k\leq 2^N}\abs{\sum_{j=2}^k D'_j}^2}
\end{equation}
and since $\pr{D'_j}_{j\geq 2}$ is a martingale 
differences sequence, we obtain by Doob's inequality and orthogonality of increments that 
\begin{equation}\label{eq:estimation_de_p'n}
p'_N\leq 
2^{1-4N/p}\sum_{j=2}^{2^N}\E{{D'_j}^2}\leq
 2^{2-4N/p}\sum_{j=2}^{2^N}\E{D_j^2\mathbf 1\ens{\abs{D_j}\leq 2^{N/p}}} .
\end{equation}
Now, we use 
\begin{align}
\E{{D}_j^2\mathbf 1\ens{\abs{{D}_j}\leq 2^{2N/p}}}&=
2\int_0^{ 2^{2N/p}}t\PP\ens{ t< \abs{{D_j}}  \leq   2^{2N/p}    }
\mathrm dt\\
&=2\int_0^{ 2^{2N/p}}t\PP\ens{ \abs{D_j}  >t   }
\mathrm dt-2^{2N/p}\PP\ens{ \abs{{D_j}}  >2^{2N/p}   }\\
 &\leq 2\int_0^{ 2^{2N/p}}t\PP\ens{\abs{{D_j}} >t   }
\mathrm dt
\end{align}
and after the substitution $s=2^{-2N/p}t$, we get 
\begin{equation}\label{eq:estimation_de_EDj_tronquee}
\E{{D}_j^2\mathbf 1\ens{\abs{{D}_j}\leq 2^{2N/p}}} \leq 2^{1+4N/p}
\int_0^{1}s\PP\ens{\abs{{D_j}} >2^{2N/p}s   }
\mathrm ds.
\end{equation}

We thus obtained the estimate 
\begin{equation}\label{eq:estimation_de_p'n_finale}
p'_N\leq 8\sum_{j=2}^{2^N} \int_0^{1}s\PP\ens{\abs{{D_j}} >2^{2N/p}s   }
\mathrm ds.
\end{equation}
 
In order to bound $p''_N:=\PP\pr{A''_N}$, we start by Markov's inequality to get 
\begin{equation}
 p''_N\leq 2^{-2N/p} \sum_{j=2}^{2^N}\E{\abs{D''_j}}
 \leq 2^{1-2N/p} \sum_{j=2}^{2^N}\E{\abs{D_j}\mathbf 1\ens{\abs{D_j}
 >2^{2N/p}}}.
\end{equation}
Since there exists a constant $C_p$ depending only on $p$ such that 
\begin{multline}
 \E{\abs{D_j}\mathbf 1\ens{\abs{D_j}
 >2^{2N/p}}}=2^{2N/p}\PP\ens{\abs{D_j}>2^{2N/p}}
 +\int_{2^{2N/p}}^{+\infty}\PP\ens{\abs{D_j}>t}\mathrm dt \\ 
 \leq  C_p \int_{2^{2N/p-1}}^{+\infty}\PP\ens{\abs{D_j}>t}\mathrm dt,
\end{multline}
we derive after the substitution $t=2^{2N/p}s$ that 
\begin{equation} \label{eq:estimation_de_p''n}
 p''_N\leq C_p 2 \sum_{j=2}^{2^N}\int_{1/2}^{+\infty}\PP\ens{\abs{D_j}>
 2^{2N/p}s}\mathrm ds.
\end{equation}

The combination of \eqref{eq:estimation_de_p'n_finale} with 
\eqref{eq:estimation_de_p''n} yields 
\begin{equation}\label{eq:borne_de_PAn_avec_Dj}
\PP\pr{A_N}\leq C_p \sum_{j=2}^{2^N}\int_{0}^{+\infty}\min\ens{1,s}  \PP\ens{\abs{D_j}>
 2^{2N/p}s}\mathrm ds.
\end{equation}
We are thus reduced to control the tail of $D_j$, which will be done by  
using Proposition~\ref{prop:inegalite_deviation_suites_iid}. Our particular 
setting permits some simplification of the involved terms.

We first observe that $D_j$ has the same distribution as $\sum_{i=1}^{j-1}
 h\pr{\eps_0,\eps_i}$ (since the vectors $\pr{\eps_1,\dots,\eps_{j-1},
 \eps_j}$ and 
 $\pr{\eps_1,\dots,\eps_{j-1},\eps_0}$ are identically distributed). Define $d_i:=h\pr{\eps_0,\eps_i}$.
 Since $\E{h\pr{\eps_0,x}}=0$ for all $x\in\R^d$, the sequence $\pr{d_i}_{i=1}^d$ is a martingale 
 differences sequence for the filtration $\pr{\mathcal G_i}_{i=1}^n$ where $\mathcal G_i$ is 
 the $\sigma$-algebra generated by $\eps_k$, $0\leq k\leq i$. We apply 
 Proposition~\ref{prop:inegalite_deviation_suites_iid} to $x= 2^{2n/p}s$ for a fixed positive $s$ 
 and $q=2p$. let $i\in\ens{1,\dots,j-1}$. 
  By Lemma~\ref{lem:propriete_esperance_cond_2_tribus} applied to 
  $Y=d_i$, $\mathcal F=\sigma\pr{\eps_0}$ and $\mathcal G=\sigma\pr{\eps_1,\dots,\eps_{i-1}}$, we have 
  \begin{equation}
  \E{\abs{d_i}^p\mid\mathcal G_{i-1}}=\E{\abs{h\pr{\eps_0,\eps_j}}^p \mid\sigma\pr{\eps_0}    }.
\end{equation}   
Using Lemma~\ref{lem:propriete_esperance_cond_fonction_d_iid} with $Y=\eps_j$, $Z=\eps_0$ and 
$f=h$, we derive that 
 \begin{equation}
  \E{\abs{d_i}^p\mid\mathcal G_{i-1}}=\E{\abs{h\pr{\eps_0,\eps_1}}^p \mid\sigma\pr{\eps_0}    }. 
\end{equation}  
Using this equality combined with the fact that the random variables $d_i$, $1\leq i\leq j-1$ have the 
same distribution as $d_1$ , one gets 
\begin{multline}
\PP\ens{\abs{D_j}>
 2^{2N/p}s}\leq c_1\pr{j-1}\int_0^1
 \PP\ens{\abs{d_1}> x 2^{2N/p}usc_2 }u^{q-1}\mathrm du\\
+ c_1 \int_0^1
 \PP\ens{\pr{j-1}^{1/p}\pr{\E{\abs{h\pr{\eps_0,\eps_1}}^p
 \mid\sigma\pr{\eps_0}    }  }^{1/p}    > 
  2^{2N/p}suc_2 }u^{q-1}\mathrm du.
\end{multline}
In view of \eqref{eq:borne_de_PAn_avec_Dj}, we derive that 
\begin{multline}\label{eq:borne_de_PAN}
\PP\pr{A_N}\leq c_12^{2N}\int_0^{+\infty}\int_0^1
 \PP\ens{\abs{d_1}> x 2^{2N/p}suc_2 }u^{2p-1}\mathrm du\min\ens{1,s}\mathrm ds\\+
 c_12^N\int_0^{+\infty}\int_0^1
 \PP\ens{ \pr{\E{\abs{h\pr{\eps_0,\eps_1}}^p \mid\sigma\pr{\eps_0}    }  }^{1/p}    > 
  2^{N/p}suc_2 }u^{2p-1}\mathrm du\min\ens{1,s}\mathrm ds.
\end{multline}
Summing over $N$, 
we get \eqref{eq:equation_cle_LGN_cas_iid} in view of the inequality
\begin{equation}\label{eq:moment_ordre_p_queue}
\sum_{N\geq 1} 2^{2N}\PP\pr{ Y>2^{2N/p}}\leq 
2\E{Y^p}
\end{equation}
for a non-negative random variable $Y$ and the convergence of the integrals $\int_0^1 u^{p-1}\mathrm du$ 
and $\int_0^{
+\infty}\min\ens{1,s}s^{-p}\mathrm ds$. This ends the proof of Proposition~\ref{prop:MZSSLNiid}.
\end{proof}
  
\subsubsection{Treatment of $R_{n,1,1}$ and $R_{n,1,2}$}  

Recall that 
\begin{multline}
H^{\pr{\ell}}_{i,j}:= 
h\pr{f_\ell\pr{V_{i,\ell}},f_\ell\pr{V_{j,\ell}}}
-\E{h\pr{f_\ell\pr{V_{i,\ell}},f_\ell\pr{V_{j,\ell}}}}\\-
\pr{h\pr{f_{\ell-1}\pr{V_{i,\ell-1}},f_{\ell-1}\pr{V_{j,\ell-1}}}
-\E{h\pr{f_{\ell-1}\pr{V_{i,\ell-1}},f_{\ell-1}\pr{V_{j,\ell-1}}}}}.
\end{multline}
 and
\begin{equation}
R_{n,1,1}:= \sum_{\ell\geq 1}
Y_{n,\ell}; \quad Y_{n,\ell}:= 
\sum_{j=\pr{4\ell+2}\ent{\frac{n}{4\ell+2}   }+1     }^n 
\sum_{i=\pr{4\ell+2}\ent{\frac{j-1}{4\ell+2}   }+1 }^{j-1}H^{\pr{\ell}}_{i,j}
\end{equation}
\begin{equation}
R_{n,1,2}:= \sum_{\ell\geq 1}Z_{n,\ell};\quad Z_{n,\ell}:=
\sum_{a=1}^{4\ell+2}\sum_{j=\pr{4\ell+2}\ent{\frac{n}{4\ell+2}   }+1     }^n 
\sum_{k=0}^{\ent{\frac{j-1}{4\ell+2}   }-1}H^{\pr{\ell}}_{\pr{4\ell+2}k+a,j}.
\end{equation}
For a fix $\ell\geq 1$, we evaluate the contribution of of $Y_{n,\ell}$ and 
$Z_{n,\ell}$.

\begin{Lemma}\label{lem:contribution_R1_LGN}
Let $\ell \geq 1$. The following inequalities hold:
\begin{equation}
\norm{\sup_{n\geq 1}\frac{1}{n^{1+1/p}} \abs{Y_{n,\ell} }}_{p,\infty}\leq 
c_p\ell^{1-1/p}\theta_{\ell,p};
\end{equation}
\begin{equation}
\norm{\sup_{n\geq 1}\frac{1}{n^{1+1/p}} \abs{Z_{n,\ell} }}_{p,\infty}\leq 
c_p\ell^{1-1/p}\theta_{\ell,p},
\end{equation}
where $c_p$ depends only on $p$.
\end{Lemma}

\begin{proof}
First observe that $Y_{n,\ell}$ is a sum of at most $\pr{4\ell+2}^2$ random 
variables whose weak-$\el^p$-norm does not exceed $\theta_{\ell,p}$ hence  
by cutting the supremum where $n$ is between two consecutive multiples of 
$4\ell+2$ gives 
\begin{align}
\norm{\sup_{n\geq 1}\frac{1}{n^{1+1/p}} \abs{Y_{n,\ell} }}_{p,\infty}&
\leq \pr{\sum_{n\geq 1}\norm{\frac{1}{n^{1+1/p}} 
\abs{Y_{n,\ell} }}_{p,\infty}^p}^{1/p}\\
&\leq \pr{4\ell+2}^{2-1-1/p}\pr{\sum_{n\geq 1}n^{-1-p}}^{1/p}\theta_{\ell,p},
\end{align}
In order to treat $Z_{n,\ell}$, we decompose it as $Z'_{n,\ell}+Z''_{n,\ell}$, 
where 
\begin{equation}
Z'_{n,\ell}=\sum_{a=1}^{4\ell+2}\sum_{j=\pr{4\ell+2}\ent{\frac{n}{4\ell+2}   }+1     }^n 
\sum_{k=0}^{\ent{\frac{j-1}{4\ell+2}   }-1}
\pr{H^{\pr{\ell}}_{\pr{4\ell+2}k+a,j}  -\E{H^{\pr{\ell}}_{\pr{4\ell+2}k+a,j}
 \mid V_{j,\ell}} }
\end{equation}
\begin{equation}
Z''_{n,\ell}=\sum_{a=1}^{4\ell+2}\sum_{j=\pr{4\ell+2}\ent{\frac{n}{4\ell+2}   }+1     }^n 
\sum_{k=0}^{\ent{\frac{j-1}{4\ell+2}   }-1}
 \E{H^{\pr{\ell}}_{\pr{4\ell+2}k+a,j}
 \mid V_{j,\ell}} .
\end{equation}
Using \eqref{eq:norm_lp_faible_supremum}, it follows that
\begin{align}
\norm{\sup_{n\geq 1}\frac{1}{n^{1+1/p}} \abs{Z'_{n,\ell} }}_{p,\infty}&\leq 
\pr{\sum_{n\geq 1}\norm{\frac{1}{n^{1+1/p}} 
\abs{Z'_{n,\ell} }}_{p,\infty}^p}^{1/p}\\
&\leq \pr{\sum_{N\geq 1} N^{-p-1}\pr{4\ell+2}^{-p-1} \pr{\sum_{n=\pr{4\ell+2}N}^{
\pr{4\ell+2}\pr{N+1}-1} \norm{Z'_{n,\ell}}_{p,\infty}   }^p  }^{1/p}
\end{align}
and for all $n$ such that $\pr{4\ell+2}N\leq n\leq \pr{4\ell+2}\pr{N+1}-1$, 
\begin{align}
\norm{Z'_{n,\ell}}_{p,\infty}
&\leq \sum_{a=1}^{4\ell+2}\sum_{j=\pr{4\ell+2}\ent{\frac{n}{4\ell+2}   }+1     }^n 
 \norm{ \sum_{k=0}^{\ent{\frac{n}{4\ell+2}   }-1} 
\pr{H^{\pr{\ell}}_{\pr{4\ell+2}k+a,j}  -\E{H^{\pr{\ell}}_{\pr{4\ell+2}k+a,j}
 \mid V_{j,\ell}} }}_{p,\infty}\\
 &\leq \sum_{a=1}^{4\ell+2}\sum_{j=\pr{4\ell+2}N  +1     }^{\pr{4\ell+2}\pr{N+1}}
 \norm{ \sum_{k=0}^{N-1} 
\pr{H^{\pr{\ell}}_{\pr{4\ell+2}k+a,j}  -\E{H^{\pr{\ell}}_{\pr{4\ell+2}k+a,j}
 \mid V_{j,\ell}} }}_{p,\infty}
\end{align}
hence 
\begin{multline}\label{eq:estimation_fct_max_LGN_R1}
\norm{\sup_{n\geq 1}\frac{1}{n^{1+1/p}} \abs{Z_{n,\ell} }}_{p,\infty}\leq\\ 
 \pr{\sum_{N\geq 1} N^{-p-1}\pr{4\ell+2}^{-p-1} \pr{ 
\sum_{a=1}^{4\ell+2}\sum_{j=\pr{4\ell+2}N  +1     }^{\pr{4\ell+2}\pr{N+1}}
 \norm{ \sum_{k=0}^{N-1} 
\pr{H^{\pr{\ell}}_{\pr{4\ell+2}k+a,j}  -\E{H^{\pr{\ell}}_{\pr{4\ell+2}k+a,j}
 \mid V_{j,\ell}} }}_{p,\infty} 
      }^p  }^{1/p}.
\end{multline}
For all fixed $j$, we notice using Lemma~\ref{lem:propriete_esperance_cond_2_tribus}
 that $\pr{H^{\pr{\ell}}_{\pr{4\ell+2}k+a,j}  -\E{H^{\pr{\ell}}_{\pr{4\ell+2}k+a,j}
 \mid V_{j,\ell}} }_{0\leq k\leq N-1}$ is a martingale differences sequence with 
 respect to the filtration $\pr{\Fca_k}_{0\leq k\leq N-1}$ where 
 \begin{equation*}
 \Fca_k:=\sigma\pr{V_{j,\ell}}\vee \sigma\pr{
V_{\pr{4\ell+2}i+a,j},i\leq k} 
 \end{equation*}
 hence by Burkholder's inequality,  
 \begin{multline}
 \norm{ \sum_{k=0}^{N-1} 
\pr{H^{\pr{\ell}}_{\pr{4\ell+2}k+a,j}  -\E{H^{\pr{\ell}}_{\pr{4\ell+2}k+a,j}
 \mid V_{j,\ell}} }}_{p,\infty}^p\\
 \leq c_p  \sum_{k=0}^{N-1} \norm{
 H^{\pr{\ell}}_{\pr{4\ell+2}k+a,j}  -\E{H^{\pr{\ell}}_{\pr{4\ell+2}k+a,j}
 \mid V_{j,\ell}} }_{p,\infty}^p\\
 \leq  2^pc_p   \sum_{k=0}^{N-1} \norm{
 H^{\pr{\ell}}_{\pr{4\ell+2}k+a,j}}_{p}^p\leq 2^pc_pN\theta_{\ell,p}^p
 \end{multline}
and plugging this estimate into \eqref{eq:estimation_fct_max_LGN_R1} gives 
 \begin{multline} 
\norm{\sup_{n\geq 1}\frac{1}{n^{1+1/p}} \abs{Z'_{n,\ell} }}_{p,\infty}\leq
 C_p\pr{\sum_{N\geq 1} N^{-p-1}\pr{4\ell+2}^{-p-1} 
 \pr{4\ell+2}^{2p} \pr{
 N\theta_{\ell,p}^p
      }^p  }^{1/p}\\
      \leq C'_p\theta_{\ell,p}\ell^{1-1/p} .
\end{multline}
In order to treat the contribution of $Z''_{n,\ell}$, we observe that 
$ \E{H^{\pr{\ell}}_{\pr{4\ell+2}k+a,j}
 \mid V_{j,\ell}}$ is independent of $k$ hence 
 \begin{equation}
Z''_{n,\ell}=\sum_{a=1}^{4\ell+2}\sum_{j=\pr{4\ell+2}\ent{\frac{n}{4\ell+2}   }+1     }^n 
\ent{\frac{n}{4\ell+2}   }
 \E{H^{\pr{\ell}}_{ a,j}
 \mid V_{j,\ell}}.
\end{equation}
Consequently,  the control of the contribution of 
$\sup_{n\geq 1}n^{-1-1/p}\abs{Z''_{n,\ell}}$ can be done thanks to 
Proposition~\ref{prop:fct_max_LGN_martingales}.

This ends the proof of Lemma~\ref{lem:contribution_R1_LGN}.
 \end{proof}
\subsubsection{Treatment of terms of the form  $\sum_u H^{\pr{\ell}}_{a,\pr{4\ell+2}u+b    }$ and 
  $\sum_u
 H^{\pr{\ell}}_{\pr{4\ell+2}u+a,\pr{4\ell+2} \ent{\frac{n}{4\ell+2}}   +b    }$
}  

\begin{Lemma}\label{lem:contribution_LGN_terms_H_a_lu+b}
For all $\ell\geq 1$ and all $a,b\in [4\ell+2]$, the following inequality holds 
\begin{equation}
\norm{\sup_{n\geq 1}   
\frac 1{n^{1+1/p}}\abs{\sum_{u=1}^{\ent{\frac{n}{4\ell+2}}    }H^{\pr{\ell}}_{a,\pr{4\ell+2}u+b    }   }
}_{p,\infty}\leq c_p \pr{4\ell+2}^{-1-1/p}\theta_{\ell,p}.
\end{equation}
\end{Lemma}

 \begin{Lemma}\label{lem:contribution_LGN_terms_H_a_lu,n+b}
For all $\ell\geq 1$ and all $a,b\in [4\ell+2]$, the following inequality holds 
\begin{equation}
\norm{\sup_{n\geq 1}   
\frac 1{n^{3/2}}\abs{\sum_{u=1}^{\ent{\frac{n}{4\ell+2}}    }H^{\pr{\ell}}_{
\pr{4\ell+2}u+a,\pr{4\ell+2}\ent{\frac{n}{4\ell+2}}+b       }}}_{p,\infty}\leq c_p \pr{4\ell+2}^{-1-1/p}\theta_{\ell,p}.
\end{equation}
\end{Lemma}

These two lemmas are the consequence of the following observations.
\begin{enumerate}
 \item We first assume that $0\leq a-b\leq 2\ell$; if not we add 
and substract $H^{\pr{\ell}}_{a,\pr{4\ell+2}\pr{u+1}+b    }$ and 
use telescoping to treat instead $\sum_{u=1}^{\ent{\frac{n}{4\ell+2}}    }H^{\pr{\ell}}_{a,\pr{4\ell+2}\pr{u+1}+b    }   $ (where we will 
apply the previous case to $\widetilde{a}=a$ and 
$\widetilde{b}=b+2\ell+1$). 
A similar method can be used to treat the sum involved in 
Lemma~\ref{lem:contribution_LGN_terms_H_a_lu,n+b}.
\item The supremum involved in the statement can be restricted 
to the integers $n$ with are a multiple of $4\ell+2$.
\item For all integer $N$ and all sequence of random variables 
$W_{i,j}$, $i,j\geq 1$, we write 
\begin{equation}
S_n:=\sum_{1\leq i<j\leq n} W_{i,j},\quad 
S'_n:=\sum_{0\leq i<j\leq n} W_{i,j}.
\end{equation}
With the choice $W_{i,j}:=  H^{\pr{\ell}}_{\pr{4\ell+2}i+a
,\pr{4\ell+2}j+b    }$, the following equality holds:
\begin{equation}
 \sum_{u=1}^{\ent{\frac{n}{4\ell+2}}    }H^{\pr{\ell}}_{a,\pr{4\ell+2}u+b    } = S'_{N}-S_{N},
\end{equation}
when $n=\pr{4\ell+2}N$. 
Since $S'_N$ and $S_N$ can be both expressed as U-statistics 
of i.i.d. data, we can use Corollary~\ref{cor:MZSSLNiid} to 
treat these terms. 

Lemma~\ref{lem:contribution_LGN_terms_H_a_lu,n+b} can be done 
in a similar way: we express this time the involved sum as 
$S_N-S_{N-1}$.
\end{enumerate}

\subsubsection{Treatment of terms of the form  $\sum_u
 H^{\pr{\ell}}_{\pr{4\ell+2}u+a,\pr{4\ell+2}u+b    }$ }  
 
 \begin{Lemma}\label{lem:contribution_LGN_terms_H_lu+a_lu+b}
For all $\ell\geq 1$ and all $a,b\in [4\ell+2]$, the following inequality holds 
\begin{equation}
\norm{\sup_{n\geq 1}   
\frac 1{n^{1+1/p}}\abs{\sum_{u=1}^{\ent{\frac{n}{4\ell+2}}    }H^{\pr{\ell}}_{
\pr{4\ell+2}u+a,\pr{4\ell+2}u+b    }   }}_{p,\infty}\leq c_p \pr{4\ell+2}^{-1-1/p}\theta_{\ell,p}.
\end{equation}
\end{Lemma}
This follows from the fact 
that $\pr{H^{\pr{\ell}}_{\pr{4\ell+2}u+a,\pr{4\ell+2}u+b}   }_{u\geq 1}$
forms a two-dependent sequence.

\subsubsection{Proof of Theorem~\ref{thm:bounded_LLN} and 
Theorem~\ref{thm:MarcinK_LLN}}

Theorem~\ref{thm:bounded_LLN} is a consequence of the combination of 
Proposition~\ref{prop:Hoeffding_decomposition}, the 
estimates of Lemmas~\ref{lem:contribution_R1_LGN}, 
\ref{lem:contribution_LGN_terms_H_a_lu+b}, \ref{lem:contribution_LGN_terms_H_lu+a_lu+b} and \ref{lem:contribution_LGN_terms_H_a_lu,n+b}.

In order to prove Theorem~\ref{thm:MarcinK_LLN}, we start 
from the decomposition 
\begin{equation}\label{eq:decomposition_U_N_somme_coupee_en_L}
U_n\pr{h,f,\pr{\eps_i}_{i\in\Z}}-\E{U_n\pr{h,f,\pr{\eps_i}_{i\in\Z}}}= 
A_{n,L}+B_{n,L},
\end{equation}
where for a fixed $L$, $A_{n,L}$ is the sum for the indexes $\ell$ smaller or 
equal to $L$ (viewing the terms associated to $V_{k,0}$ as the corresponding 
ones for $\ell=0$) and $B_{n,L}$ the remaining term. We have to prove that 
for each positive $\varepsilon$, 
\begin{equation}
\lim_{N\to+\infty} \PP\ens{\sup_{n\geq N}
\frac 1{n^{1+1/p}}\abs{U_n\pr{h,f,\pr{\eps_i}_{i\in\Z}}-\E{U_n\pr{h,f,\pr{\eps_i}_{i\in\Z}}} }>2\eps
}=0.
\end{equation}
Using \eqref{eq:decomposition_U_N_somme_coupee_en_L} and the 
fact that $A_{n,L}$ consists of sums of terms which 
can be treated by Lemmas~\ref{lem:contribution_R1_LGN}, 
\ref{lem:contribution_LGN_terms_H_a_lu+b}, \ref{lem:contribution_LGN_terms_H_lu+a_lu+b} and \ref{lem:contribution_LGN_terms_H_a_lu,n+b}, we derive that 
for all fixed $L$, 
\begin{equation}
\limsup_{N\to+\infty} \PP\ens{\sup_{n\geq N}
\frac 1{n^{1+1/p}}\abs{U_n\pr{h,f,\pr{\eps_i}_{i\in\Z}}-\E{U_n\pr{h,f,\pr{\eps_i}_{i\in\Z}}} }>2\eps
}\leq \limsup_{N\to+\infty} \PP\ens{\sup_{n\geq N}
\frac 1{n^{1+1/p}}\abs{B_{n,L} }>\eps}.
\end{equation}
Bounding the latter probability by 
$\eps^{-p}\norm{\sup_{n\geq 1}
\frac 1{n^{1+1/p}}\abs{B_{n,L} }}_{p,w}^p$ and using Lemmas~\ref{lem:contribution_R1_LGN}, 
\ref{lem:contribution_LGN_terms_H_a_lu+b}, \ref{lem:contribution_LGN_terms_H_lu+a_lu+b} and \ref{lem:contribution_LGN_terms_H_a_lu,n+b}, we can 
see that for some constant $C$ depending only on $p$, 
\begin{equation}
\norm{\sup_{n\geq 1}
\frac 1{n^{1+1/p}}\abs{B_{n,L} }}_{p,w}\leq C \sum_{\ell\geq L}
\ell^2\theta_{\ell,p},
\end{equation}
which can be made arbitrarily small.

\subsection{Proof of Theorem~\ref{thm:LLI_Bernoulli}}

Like for the results on the strong law of larger numbers, we have to control 
the contribution of the extra-terms in the decomposition obtained in 
Proposition~\ref{prop:Hoeffding_decomposition}.

\subsubsection{Treatment of $R_{n,1,1}$ and $R_{n,1,2}$}  

Recall that 
\begin{multline}
H^{\pr{\ell}}_{i,j}:= 
h\pr{f_\ell\pr{V_{i,\ell}},f_\ell\pr{V_{j,\ell}}}
-\E{h\pr{f_\ell\pr{V_{i,\ell}},f_\ell\pr{V_{j,\ell}}}}\\-
\pr{h\pr{f_{\ell-1}\pr{V_{i,\ell-1}},f_{\ell-1}\pr{V_{j,\ell-1}}}
-\E{h\pr{f_{\ell-1}\pr{V_{i,\ell-1}},f_{\ell-1}\pr{V_{j,\ell-1}}}}}.
\end{multline}
 and
\begin{equation}
R_{n,1,1}:= \sum_{\ell\geq 1}
Y_{n,\ell};\quad Y_{n,\ell}:= 
\sum_{j=\pr{4\ell+2}\ent{\frac{n}{4\ell+2}   }+1     }^n 
\sum_{i=\pr{4\ell+2}\ent{\frac{j-1}{4\ell+2}   }+1 }^{j-1}H^{\pr{\ell}}_{i,j}
\end{equation}
\begin{equation}
R_{n,1,2}:= \sum_{\ell\geq 1}Z_{n,\ell};\quad Z_{n,\ell}:=
\sum_{a=1}^{4\ell+2}\sum_{j=\pr{4\ell+2}\ent{\frac{n}{4\ell+2}   }+1     }^n 
\sum_{k=0}^{\ent{\frac{j-1}{4\ell+2}   }-1}H^{\pr{\ell}}_{\pr{4\ell+2}k+a,j}.
\end{equation}
For a fix $\ell\geq 1$, we evaluate the contribution of of $Y_{n,\ell}$ and 
$Z_{n,\ell}$.

\begin{Lemma}\label{lem:contribution_R1_LLI}
Let $\ell \geq 1$. The following inequalities hold.
\begin{equation}
\norm{\sup_{n\geq 1}\frac{1}{n^{3/2}\sqrt{\LL{n}}} \abs{Y_{n,\ell} }}_{p,\infty}\leq 
c_p\ell^{1-1/p}\theta_{\ell,p}
\end{equation}
\begin{equation}
\norm{\sup_{n\geq 1}\frac{1}{n^{3/2}\sqrt{\LL{n}}} \abs{Z_{n,\ell} }}_{p,\infty}\leq c_p\ell^2\theta_{\ell,p}.
\end{equation}
\end{Lemma}

\begin{proof}
First observe that $Y_{n,\ell}$ is a sum of at most $\pr{4\ell+2}^2$ random 
variables whose weak-$\el^p$-norm does not exceed $\theta_{\ell,p}$
hence  
by cutting the supremum where $n$ is between two consecutive multiples of 
$4\ell+2$ gives  
\begin{align}
\norm{\sup_{n\geq 1}\frac{1}{n^{3/2}\sqrt{\LL{n}}} \abs{Y_{n,\ell} }}_{p,\infty}&
\leq \pr{\sum_{n\geq 1}\norm{\frac{1}{n^{3/2}\sqrt{\LL{n}}} 
\abs{Y_{n,\ell} }}_{p,\infty}^p}^{1/p}\\
&\leq c\pr{4\ell+2}^{1-1/p}\pr{\sum_{n\geq 1}n^{-3p/2}}^{1/p}\theta_{\ell,p}.
\end{align} 
In order to treat $Z_{n,\ell}$, we decompose it as $Z'_{n,\ell}+Z''_{n,\ell}$, 
where 
\begin{equation}
Z'_{n,\ell}=\sum_{a=1}^{4\ell+2}\sum_{j=\pr{4\ell+2}\ent{\frac{n}{4\ell+2}   }+1     }^n 
\sum_{k=0}^{\ent{\frac{j-1}{4\ell+2}   }-1}
\pr{H^{\pr{\ell}}_{\pr{4\ell+2}k+a,j}  -\E{H^{\pr{\ell}}_{\pr{4\ell+2}k+a,j}
 \mid V_{j,\ell}} }
\end{equation}
\begin{equation}
Z''_{n,\ell}=\sum_{a=1}^{4\ell+2}\sum_{j=\pr{4\ell+2}\ent{\frac{n}{4\ell+2}   }+1     }^n 
\sum_{k=0}^{\ent{\frac{j-1}{4\ell+2}   }-1}
 \E{H^{\pr{\ell}}_{\pr{4\ell+2}k+a,j}
 \mid V_{j,\ell}} .
\end{equation}
We first use \eqref{eq:norm_lp_faible_supremum} to get 
\begin{align}
\norm{\sup_{n\geq 1}\frac{1}{n^{3/2}\sqrt{\LL{n}}}
 \abs{Z'_{n,\ell} }}_{p,\infty}&\leq 
\pr{\sum_{n\geq 3}\norm{\frac{1}{n^{3/2}\sqrt{\LL{n}}} 
\abs{Z'_{n,\ell} }}_{p,\infty}^p}^{1/p}\\
&\leq \pr{\sum_{N\geq 1} N^{-3p/2}\pr{4\ell+2}^{-3p/2} \pr{\sum_{n=\pr{4\ell+2}N}^{
\pr{4\ell+2}\pr{N+1}-1} \norm{Z'_{n,\ell}}_{p,\infty}   }^p  }^{1/p}
\end{align}
and for all $n$ such that $\pr{4\ell+2}N\leq n\leq \pr{4\ell+2}\pr{N+1}-1$, 
\begin{align}
\norm{Z'_{n,\ell}}_{p,\infty}
&\leq \sum_{a=1}^{4\ell+2}\sum_{j=\pr{4\ell+2}\ent{\frac{n}{4\ell+2}   }+1     }^n 
 \norm{ \sum_{k=0}^{\ent{\frac{n}{4\ell+2}   }-1} 
\pr{H^{\pr{\ell}}_{\pr{4\ell+2}k+a,j}  -\E{H^{\pr{\ell}}_{\pr{4\ell+2}k+a,j}
 \mid V_{j,\ell}} }}_{p,\infty}\\
 &\leq \sum_{a=1}^{4\ell+2}\sum_{j=\pr{4\ell+2}N  +1     }^{\pr{4\ell+2}\pr{N+1}}
 \norm{ \sum_{k=0}^{N-1} 
\pr{H^{\pr{\ell}}_{\pr{4\ell+2}k+a,j}  -\E{H^{\pr{\ell}}_{\pr{4\ell+2}k+a,j}
 \mid V_{j,\ell}} }}_{p,\infty}
\end{align}
hence 
\begin{multline}\label{eq:estimation_fct_max_LLI_R1}
\norm{\sup_{n\geq 1}\frac{1}{n^{3/2}\sqrt{\LL{n}}} \abs{Z'_{n,\ell} }}_{p,\infty}
^p\leq\\ 
 \sum_{N\geq 1} N^{-p-1}\pr{4\ell+2}^{-3p/2} \pr{ 
\sum_{a=1}^{4\ell+2}\sum_{j=\pr{4\ell+2}N  +1     }^{\pr{4\ell+2}\pr{N+1}}
 \norm{ \sum_{k=0}^{N-1} 
\pr{H^{\pr{\ell}}_{\pr{4\ell+2}k+a,j}  -\E{H^{\pr{\ell}}_{\pr{4\ell+2}k+a,j}
 \mid V_{j,\ell}} }}_{p,\infty} 
      }^p   .
\end{multline}
For all fixed $j$, we notice using Lemma~\ref{lem:propriete_esperance_cond_2_tribus}
 that $\pr{H^{\pr{\ell}}_{\pr{4\ell+2}k+a,j}  -\E{H^{\pr{\ell}}_{\pr{4\ell+2}k+a,j}
 \mid V_{j,\ell}} }_{0\leq k\leq N-1}$ is a martingale differences sequence with 
 respect to the filtration $\pr{\Fca_k}_{0\leq k\leq N-1}$ where 
 \begin{equation*}
 \Fca_k:=\sigma\pr{V_{j,\ell}}\vee \sigma\pr{
V_{\pr{4\ell+2}i+a,j},i\leq k} 
 \end{equation*}
 hence 
 \begin{multline}
 \norm{ \sum_{k=0}^{N-1} 
\pr{H^{\pr{\ell}}_{\pr{4\ell+2}k+a,j}  -\E{H^{\pr{\ell}}_{\pr{4\ell+2}k+a,j}
 \mid V_{j,\ell}} }}_{p,\infty}^p\\ \leq 
  \norm{ \sum_{k=0}^{N-1} 
\pr{H^{\pr{\ell}}_{\pr{4\ell+2}k+a,j}  -\E{H^{\pr{\ell}}_{\pr{4\ell+2}k+a,j}
 \mid V_{j,\ell}} }}_{2}^p\\
 \leq \pr{\sum_{k=0}^{N-1} \norm{
\pr{H^{\pr{\ell}}_{\pr{4\ell+2}k+a,j}  -\E{H^{\pr{\ell}}_{\pr{4\ell+2}k+a,j}
 \mid V_{j,\ell}} }}_{2}^2 }^{p/2} 
 \leq   c_pN^{p/2}\theta_{\ell,2}^p
 \end{multline}
and plugging this estimate into \eqref{eq:estimation_fct_max_LLI_R1} gives 
 \begin{multline} 
\norm{\sup_{n\geq 1}\frac{1}{n^{3/2}\sqrt{\LL{n}}}
 \abs{Z'_{n,\ell} }}_{p,\infty}\leq
 C_p\pr{\sum_{N\geq 1} N^{-1-p/2}\pr{4\ell+2}^{-3p/2+2} 
 \pr{4\ell+2}^{2p}  
 N\theta_{\ell,p}^p
       }^{1/p}\\
      \leq C'_p\theta_{\ell,2}\ell^{1/2} .
\end{multline}
We control the contribution of $Z''_{n,\ell}$ by noticing that 
$\E{H^{\pr{\ell}}_{\pr{4\ell+2}k+a,j}
 \mid V_{j,\ell}}$ is actually independent on $k$.

This ends the proof of Lemma~\ref{lem:contribution_R1_LLI}.
 \end{proof}
\subsubsection{Treatment of terms of the form  $\sum_u H^{\pr{\ell}}_{a,\pr{4\ell+2}u+b    }$ and $\sum_u
 H^{\pr{\ell}}_{\pr{4\ell+2}u+a,\pr{4\ell+2} \ent{\frac{n}{4\ell+2}}   +b    }$ }  

\begin{Lemma}\label{lem:contribution_LLI_terms_H_a_lu+b}
For all $\ell\geq 1$ and all $a,b\in [4\ell+2]$, the following inequality holds 
\begin{equation}
\norm{\sup_{n\geq 1}   
\frac 1{n^{3/2}\sqrt{\LL{n}}}\abs{\sum_{u=1}^{\ent{\frac{n}{4\ell+2}}    }H^{\pr{\ell}}_{a,\pr{4\ell+2}u+b    }   }
}_{p,\infty}\leq c_p \pr{4\ell+2}^{-3/2}\theta_{\ell,p}.
\end{equation}
\end{Lemma}

\begin{Lemma}\label{lem:contribution_LLI_terms_H_a_lu,n+b}
For all $\ell\geq 1$ and all $a,b\in [4\ell+2]$, the following inequality holds 
\begin{equation}
\norm{\sup_{n\geq 1}   
\frac 1{n^{3/2}\sqrt{\LL{n}}}
\abs{\sum_{u=1}^{\ent{\frac{n}{4\ell+2}}    }H^{\pr{\ell}}_{
\pr{4\ell+2}u+a,\pr{4\ell+2}\ent{\frac{n}{4\ell+2}}+b       }}}_{p,\infty}\leq c_p \pr{4\ell+2}^{-1-1/p}\theta_{\ell,p}.
\end{equation}
\end{Lemma}

The proof follows exactly the same idea as the proof of 
Lemmas~\ref{lem:contribution_LGN_terms_H_a_lu+b} and 
\ref{lem:contribution_LGN_terms_H_a_lu,n+b}, where the use of 
Corollary~\ref{cor:MZSSLNiid} is replaced by that of 
Proposition~\ref{prop:fct_max_LLI_fct_de_ell_iid}.

\subsubsection{Treatment of terms of the form  $\sum_u
 H^{\pr{\ell}}_{\pr{4\ell+2}u+a,\pr{4\ell+2}u+b    }$ }  
 
 \begin{Lemma}\label{lem:contribution_LLI_terms_H_lu+a_lu+b}
For all $\ell\geq 1$ and all $a,b\in [4\ell+2]$, the following inequality holds 
\begin{equation}
\norm{\sup_{n\geq 1}   
\frac 1{n^{3/2}}\abs{\sum_{u=1}^{\ent{\frac{n}{4\ell+2}}    }H^{\pr{\ell}}_{
\pr{4\ell+2}u+a,\pr{4\ell+2}u+b    }   }}_{p,\infty}\leq c_p \pr{4\ell+2}^{-3/2}\theta_{\ell,p}.
\end{equation}
\end{Lemma}
This follows from the fact 
that $\pr{H^{\pr{\ell}}_{\pr{4\ell+2}u+a,\pr{4\ell+2}u+b}   }_{u\geq 1}$
forms a two-dependent sequence.

\subsection{Proof of Theorem~\ref{thm:TLC}}
  
Let us explain the idea of the proof. The convergence is essentially 
due to the partial sums of strictly stationnary sequence $\pr{Y_k}_{k\geq 1}$, 
where 
\begin{equation}
Y_k= Y_{k,0}+\sum_{\ell \geq 1}Y_{k,\ell}-Y_{k,\ell-1}
\end{equation}   
and 
\begin{equation}
Y_{k,\ell}:=\E{h\pr{f_\ell\pr{V_{k,\ell}},f_{\ell}\pr{  V'_{0,\ell}  }} \mid V_{k,\ell}  }
-\E{h\pr{f_\ell\pr{V_{k,\ell}},f_{\ell}\pr{  V'_{0,\ell}  }}   }.
\end{equation}
We can establish the convergence of $\pr{n^{-1/2}\sum_{k=1}^nY_k}_{n\geq 1}$
by showing the convergence of   $\pr{n^{-1/2}\sum_{k=1}^nY_{k,L}}_{n\geq 1}$
  for a fixed $L$ and by controlling the remainder. We have to prove that all the terms
  in the decomposition obtained in Proposition~\ref{prop:Hoeffding_decomposition} 
  converge to zero in probability. To sum up, we start by writing 
  \begin{multline}
 \frac 1{n^{3/2}} U_n\pr{h,f,\pr{\eps_i}_{i\in\Z}}-\E{U_n\pr{h,f,\pr{\eps_i}_{i\in\Z}}}=\frac 1{\sqrt n}\sum_{k=1}^nY_k+ 
U^{\ind}_n\pr{h^{\pr{0}},\pr{\eps_i}_i}+
 \\
+\frac 1{n^{3/2}}\sum_{\ell\geq 1}\sum_{a,b\in [4\ell+2]} U_{\ent{\frac{n}{2\ell}}   }^{\ind}\pr{h_{a,b}^{\pr{\ell}}    ,
\pr{\eps_i^{a,b}}   }+R_{n,1,1}+R_{n,1,2}+
\sum_{k=2}^7R_{n,k},
\end{multline}
where for each $\ell\geq 1$ and all  $a,b\in [4\ell+2]$, 
the U-statistic  $U_{\ent{\frac{n}{2\ell}}   }^{\ind}\pr{h_{a,b}^{\pr{\ell}}    
,\pr{\eps_i^{a,b}}   }$ has independent data and 
is degenerated, and the remainder terms are defined as 
\begin{equation}
R_{n,1,1}:= \frac 1{n^{3/2}}\sum_{\ell\geq 1}\sum_{j=\pr{4\ell+2}\ent{\frac{n}{4\ell+2}   }+1     }^n 
\sum_{i=\pr{4\ell+2}\ent{\frac{j-1}{4\ell+2}   }+1 }^{j-1}H^{\pr{\ell}}_{i,j}
\end{equation}
\begin{equation}
R_{n,1,2}:=\frac 1{n^{3/2}} \sum_{\ell\geq 1}
\sum_{a=1}^{4\ell+2}\sum_{j=\pr{4\ell+2}\ent{\frac{n}{4\ell+2}   }+1     }^n 
\sum_{k=0}^{\ent{\frac{j-1}{4\ell+2}   }-1}H^{\pr{\ell}}_{\pr{4\ell+2}k+a,j};
\end{equation}
\begin{equation}
R_{n,2}:=\frac 1{n^{3/2}}\sum_{\ell\geq 1}\sum_{u=0}^{\ent{\frac{n}{4\ell+2}}-1}
\sum_{\substack{ a,b\in [ 4\ell+2]\\ a<b  }   }H^{\pr{\ell}}_{u\pr{4\ell+2}+a,u\pr{4\ell+2}+b}
\end{equation}
 \begin{equation}
R_{n,3}:=\frac 1{n^{3/2}} \sum_{\ell\geq 1}\sum_{v=1}^{\ent{\frac{n}{4\ell+2}}-1}
\sum_{\substack{ a,b\in [ 4\ell+2]\\ 0\leq a-b\leq \pr{2\ell+1}-1   }   }
\pr{H^{\pr{\ell}}_{ a,v\pr{4\ell+2}+b}+H^{\pr{\ell}}_{ b,v\pr{4\ell+2}+a}}
\end{equation}
 \begin{equation}
R_{n,4}= \frac 1{n^{3/2}}\sum_{\ell\geq 1}\sum_{\substack{ a,b\in [ 4\ell+2]\\ \pr{2\ell+1}\leq a-b\leq  \pr{4\ell+2}-1   }   } \sum_{v=1}^{\ent{\frac{n}{4\ell+2}}-1} 
\pr{ H^{\pr{\ell}}_{ a,\pr{v+1}\pr{4\ell+2}+b}  
+H^{\pr{\ell}}_{ \pr{4\ell+2}+b,v\pr{4\ell+2}+a}}
\end{equation}
 \begin{equation}
R_{n,5}=\frac 1{n^{3/2}} \sum_{\ell\geq 1}\sum_{\substack{ a,b\in [ 4\ell+2]\\ \pr{2\ell+1}\leq a-b\leq  \pr{4\ell+2}-1   }   } \sum_{ u=0}^{ \ent{\frac{n}{4\ell+2}}-2} 
\pr{H^{\pr{\ell}}_{u\pr{4\ell+2}+a,\pr{u+1}\pr{4\ell+2}+b}-H^{\pr{\ell}}_{u\pr{4\ell+2}+a,2m\pr{2\ell+1}+b}}
\end{equation}
 \begin{equation}
R_{n,6}= \frac 1{n^{3/2}} \sum_{\ell\geq 1}\sum_{\substack{ a,b\in [ 4\ell+2]
\\ \pr{2\ell+1}\leq a-b\leq  \pr{4\ell+2}-1   }   } 
\sum_{v=1}^{\ent{\frac{n}{4\ell+2}}-1} 
\pr{ H^{\pr{\ell}}_{ b,v\pr{4\ell+2}+a}- H^{\pr{\ell}}_{v\pr{4\ell+2}+b,v\pr{4\ell+2}+a}},
\end{equation} 
\begin{multline}
R_{n,7}:=  \frac 1{n^{3/2}}\sum_{\ell \geq1} R_{n,7}^{\pr{\ell}},\\ 
R_{n,7}^{\pr{\ell}}:=
\pr{4\ell+1} \ent{\frac{n}{4\ell+1}}
 \sum_{k=1}^{\pr{4\ell+2} \ent{\frac{n}{4\ell+2}}+1 }
\pr{Y_{k,\ell}-Y_{k,\ell-1}}-n\sum_{k=1}^n\pr{Y_{k,\ell}-Y_{k,\ell-1}}
\end{multline}

with 
\begin{multline}
H^{\pr{\ell}}_{i,j}:= 
h\pr{f_\ell\pr{V_{i,\ell}},f_\ell\pr{V_{j,\ell}}}
-\E{h\pr{f_\ell\pr{V_{i,\ell}},f_\ell\pr{V_{j,\ell}}}}\\-
\pr{h\pr{f_{\ell-1}\pr{V_{i,\ell-1}},f_{\ell-1}\pr{V_{j,\ell-1}}}
-\E{h\pr{f_{\ell-1}\pr{V_{i,\ell-1}},f_{\ell-1}\pr{V_{j,\ell-1}}}}}.
\end{multline}
We then follow the steps~:
\begin{enumerate}
\item we show that $\pr{n^{-1/2}\sum_{k=1}^nY_k}_{n\geq 1}$ converges 
to a centered normal distribution with variance 
\item We show the convergence in probability to zero of all the terms 
$R_{n,1,1}$, $R_{n,1,2}$, $R_{n,k}$, $2\leq k\leq 7$.
\end{enumerate}

\subsubsection{Convergence of $\pr{n^{-1/2}\sum_{k=1}^nY_k}_{n\geq 1}$}  

We use Theorem 4.2 in \cite{MR0233396}, which states the following.
For $L,n\in\mathbb N$, $Z_n,Z_{n, L}$, $W_L$ and $Z$ are real-valued random variables defined on a common probability space $\left(\Omega,\mathcal F,\PP\right)$. We assume that 
\begin{enumerate}
\item for all $L\in\mathbb N$, $Z_{n,L}\to W_L$ in distribution as $n\to \infty$;
\item $W_L\to Z$ in distribution as $L\to \infty$, and 
\item for each $\varepsilon>0$, $\lim_{L\to \infty}\limsup_{n\to \infty}\PP\left\{\abs{Z_{n,L}-Z_n}>\varepsilon\right\}=0$.
\end{enumerate}
Then $Z_n\to Z$ in distribution as $n\to \infty$.

We will apply the result in the following setting: 
\begin{equation}
Z_{n,L}:=\frac 1{\sqrt n}\sum_{k=1}^nY_{k,L},
\end{equation}
$W_L$ a centered normal variable with variance 
$\sigma_L^2:=\sum_{k=-L}^L\operatorname{Cov}
\pr{Y_{0,L},Y_{k,L}}$ 
and $W$ a centered normal variable with variance $\sigma^2 
=\sum_{k\in\Z}\operatorname{Cov}
\pr{Y_{0},Y_{k}}$. 

The first item follows from the central limit theorem for $\pr{2L+1}$-dependent 
random variables; the second one from the convergence of $\pr{\sigma_L^2}_{L\geq 1}$ 
to $\sigma^2$, which can be seen by writing 
$Y_{k,L}-Y_k=\sum_{\ell\geq L}\pr{Y_{k,\ell}-Y_{k,\ell+1}}$ and using 
Cauchy-Schwarz inequality:
\begin{multline}
 \abs{\sum_{k=-L}^L\operatorname{Cov}\pr{Y_{0,L},Y_{k,L}} 
 - \operatorname{Cov}\pr{Y_{0 },Y_{k}}        }\leq 
 \sum_{k=-L}^L\abs{\operatorname{Cov}\pr{Y_{0,L},Y_{k,L}-Y_k}       }+
 \sum_{k=-L}^L\abs{\operatorname{Cov}\pr{Y_{0,L}-Y_0,Y_{k}}} \\
 \leq 2L\norm{Y_{0,L}}_2\sum_{\ell\geq L}\theta_{\ell,2}
 +2L\norm{Y_{k}}_2\sum_{\ell\geq L}\theta_{\ell,2}
\end{multline}
and the quantities $\norm{Y_{0,L}}_2$ and $\norm{Y_{k}}_2$ 
are bounded independently of $L$ and $k$.

For the third item, we start from Chebytchev's inequality:
\begin{equation}
\PP\left\{\abs{Z_{n,L}-Z_n}>\varepsilon\right\}\leq 
\eps^{-2}\norm{Z_{n,L}-Z_n}_2^2
\end{equation}
and 
\begin{equation}
\norm{Z_{n,L}-Z_n}_2\leq \sum_{\ell\geq L}
\norm{\frac 1{\sqrt n} \sum_{k=1}^n \pr{Y_{k,\ell}-Y_{k,\ell-1}}   }_2
\end{equation}
and $\norm{\frac 1{\sqrt n} \sum_{k=1}^n \pr{Y_{k,\ell}-Y_{k,\ell-1}}   }_2$ 
does not exceed a constant times $\ell^{1/2} \theta_{\ell,2}$ hence 
\begin{equation}
\limsup_{n\to +\infty}\PP\left\{\abs{Z_{n,L}-Z_n}>\varepsilon\right\}\leq 
\eps^{-2}\pr{\sum_{\ell\geq L} \ell^{1/2} \theta_{\ell,2}    }^2.
\end{equation}

\subsubsection{Convergence in probability of $R_{n,1,1}$, $R_{n,1,2}$ 
and $R_{n,k}$, $2\leq k\leq 6$
to $0$}

Observe that 
\begin{equation}
\E{\abs{R_{n,1,1}}}\leq \frac 1{n^{3/2}}\sum_{\ell\geq 1}\sum_{j=\pr{4\ell+2}\ent{\frac{n}{4\ell+2}   }+1     }^n 
\sum_{i=\pr{4\ell+2}\ent{\frac{j-1}{4\ell+2}   }+1 }^{j-1}\norm{H^{\pr{\ell}}_{i,j}}_1,
\end{equation}
that $\norm{H^{\pr{\ell}}_{i,j}}_1\leq \theta_{\ell,1}$ and that the 
number of terms in the two inner sums is of order $\ell^2$ hence 
\begin{equation}
\E{\abs{R_{n,1,1}}}\leq \frac 1{n^{3/2}} \sum_{\ell\geq 1}\ell^2\theta_{\ell,1}.
\end{equation}

We use the same method for the terms $R_{n,1,2}$ and $R_{n,k}$, $2\leq k\leq 6$.

\subsubsection{Treatment of $R_{n,7}$} 

First we rewrite $R_{n,7}$ as a double sum, namely, 
as
\begin{equation}
R_{n,7}=\frac 1{n^{3/2}}\sum_{\ell\geq 1}\sum_{k\geq 1}
\pr{Y_{k,\ell}-Y_{k,\ell-1}}c_{n,k,\ell}
\end{equation}
where 
\begin{equation}
c_{n,k,\ell}= \pr{4\ell+2}\ent{\frac{n}{4\ell+2}}\left[k\leq  \pr{4\ell+2}\ent{\frac{n}{4\ell+2}}+1     \right]-n \left[k\leq n\right],
\end{equation}
with the notation $[P]=1$ if the assertion $P$ holds and $0$ otherwise.
Write $n$ as $\pr{4\ell+2}N+r$, where $0\leq r\leq 4\ell+1$. Then 
\begin{equation}
c_{n,k,\ell}= -q \left[k\leq  N\pr{4\ell+2}+1     \right]-n \left[N\pr{4\ell+2}
   \leq k\leq n\right],
\end{equation}
hence 
\begin{equation}
\abs{R_{n,7}}\leq 
\frac 1{n^{3/2}}\sum_{\ell\geq 1}\ell \abs{
\sum_{k=1}^{\pr{4\ell+2}\ent{\frac{n}{4\ell+2}}+1}
\pr{Y_{k,\ell}-Y_{k,\ell-1}}}   
+\frac 1{n^{1/2}}  \sum_{\ell\geq 1}
 \abs{
\sum_{\pr{4\ell+2}\ent{\frac{n}{4\ell+2}}+1}^{n}
\pr{Y_{k,\ell}-Y_{k,\ell-1}}}
\end{equation}
and taking the expectation give that 
\begin{equation}
\norm{R_{n,7}}_1\leq 2n^{-1/2}\sum_{\ell\geq 1}\ell\theta_{\ell,1}.
\end{equation}
Acknowledgement: this research supported by the German National Academic Foundation and Collaborative Research Center SFB 823   "Statistical modelling of nonlinear dynamic processes".
\begin{appendices}
\section{Appendix}

In this appendix, we collect some fact about partial sums of martingales or 
functional of independent sequences that we will need in the proof. 

The first is a probability inequatlity for martingales;
 Proposition\ref{prop:fct_max_LGN_martingales} and\ref{prop:fct_max_LGN_fct_de_ell_iid} give a control of the maximal function 
 involved in the strong law of large numbers, respectively for martingales 
 and functionals of a fixed number of i.i.d. random variables. We end 
 the Appendix by two lemmas on conditional expectation.
   
 \begin{Proposition}[Theorem 1.3 
 in \cite{MR4046858}]\label{prop:inegalite_deviation_suites_iid}
 Let $1<p<2$ and $q>p$. Then there exists constants $c_1$ and $c_2$ depending only on 
 $\ens{p,q}$ such that if $\pr{d_i}_{i= 1}^n$ is a martingale differences sequence 
 with respect to a filtration $\pr{\mathcal F_i}_{i=1}^n$, then for each integer $n$ 
 and each positive $x$, 
 \begin{multline}
 \PP\ens{\abs{\sum_{i=1}^nd_i}>x }\leq c_1\sum_{i=1}^n\int_0^1
 \PP\ens{\abs{d_i}> xuc_2 }u^{q-1}\mathrm du\\
+ c_1 \int_0^1
 \PP\ens{\pr{\sum_{i=1}^n\E{\abs{d_i}^p\mid\mathcal F_{i-1}}}^{1/p}    > xuc_2 }u^{q-1}\mathrm du.
 \end{multline}
 \end{Proposition}

The next Proposition gives a control of the maximal function involved 
in the strong law of large numbers. A control on the $r$-th moment for 
$r<p$ was obtained in the proof of Theorem~4.1 in \cite{MR668549}. 
The control on the weak-$\mathbb L^p$-moment was explicitely 
established in \cite{MR3322323}, p.~324, under a stationarity 
assumption, but the proof work for martingale with identically distributed 
increments.

\begin{Proposition}\label{prop:fct_max_LGN_martingales}
Let $\pr{d_j}_{j\geq 1}$ be an identically distributed
 martingale differences sequence 
  with respect to the 
filtration $\pr{\Fca_j}_{j\geq 0}$. Then for all $1<p<2$, 
\begin{equation}
\norm{\sup_{n\geq 1}\frac 1{n^{1/p}}  \abs{\sum_{j=1}^nd_j}}_{p,\infty}\leq c_p 
\norm{d_1}_p.
\end{equation}
\end{Proposition}

\begin{Proposition}\label{prop:fct_max_LGN_fct_de_ell_iid}
Let $\ell\geq 1$ and $\pr{X_j}_{j\geq 1}$ be a sequence such that 
$X_j=f\pr{\pr{\eps_u}_{u=j-\ell}^{j+\ell}  }$, where $\pr{\eps_u}_{u\in\Z}$ 
is i.i.d. and $X_1$ is centered. Then for all $1<p<2$, 
\begin{equation}
\norm{\sup_{n\geq 1}\frac 1{n^{1/p}}  \abs{\sum_{j=1}^nX_j}}_{p,\infty}\leq c_p 
\ell^{1-1/p}\norm{X_1}_p.
\end{equation}
\end{Proposition}

\begin{proof}
We first notice that 
\begin{equation}
\sup_{n\geq 1}\frac 1{n^{1/p}}  \abs{\sum_{j=1}^nX_j}
\leq \sum_{a\in [4\ell+2]}\sup_{n\geq 1}\frac 1{n^{1/p}}  
\abs{\sum_{k=1}^{\ent{\frac{n}{4\ell+2}}  }X_{\pr{4\ell+2}k+a }   }
\end{equation}
and we apply Proposition~\ref{prop:fct_max_LGN_martingales} to the 
sequences $\pr{X_{\pr{4\ell+2}k+a }}_{k\geq 1}$ for all fixed $a\in [4\ell+2]$.
\end{proof}

The next Proposition give a control of the weak-$\mathbb L^p$-norm of 
the maximum function involved in the bounded law of the iterated logarithms 
for partial sum of stationary sequences. When the involved 
sequence is i.i.d. and centered, this reduces to Théorème~1 in \cite{MR0501237}. 
This can be extended to the context of the functional of $2\ell+1$ i.i.d. 
random variables by the same method as in the proof of 
Proposition~\ref{prop:fct_max_LLI_fct_de_ell_iid}.

\begin{Proposition}\label{prop:fct_max_LLI_fct_de_ell_iid}
Let $\ell\geq 1$ and $\pr{X_j}_{j\geq 1}$ be a sequence such that 
$X_j=f\pr{\pr{\eps_u}_{u=j-\ell}^{j+\ell}  }$, where $\pr{\eps_u}_{u\in\Z}$ 
is i.i.d. and $X_1$ is centered. Then for all $1<p<2$, 
\begin{equation}
\norm{\sup_{n\geq 1}\frac 1{\sqrt{n\LL{n}}}  \abs{\sum_{j=1}^nX_j}}_{p,\infty}\leq c_p 
\ell^{1/2}\norm{X_1}_2.
\end{equation}
\end{Proposition}

\begin{Lemma}[Proposition~2 p.~1693 in \cite{MR3222815}]\label{lem:propriete_esperance_cond_2_tribus}
Let $Y$ be a real-valued random variable and let $\mathcal F$ and $\mathcal G$ be 
two sub-$\sigma$-algebra such that $\mathcal G$ is independent of the 
$\sigma$-algebra generated by $Y$ and $\mathcal F$. Then 
\begin{equation}
\E{Y\mid \mathcal F\vee \mathcal G}=\E{Y\mid \mathcal F}.
\end{equation}
\end{Lemma}

Finally, the next lemma is well-known.

  \begin{Lemma}\label{lem:propriete_esperance_cond_fonction_d_iid}
  Let $Y$ and $Z$ be two independent random variables with values in $\R^d$. 
 Let $f\colon \R^d\times \R^d\to \R$ be a measurable function. Then 
 \begin{equation}
 \E{f\pr{Y,Z}\mid \sigma\pr{Z} }=g\pr{Z},
\end{equation}   
where $g\colon \R^d\to \R$ is defined by $g\pr{z}=\E{f\pr{Y,z}}$.
  \end{Lemma}
\end{appendices} 
\def\polhk\#1{\setbox0=\hbox{\#1}{{\o}oalign{\hidewidth
  \lower1.5ex\hbox{`}\hidewidth\crcr\unhbox0}}}\def\cprime{$'$}
  \def\polhk#1{\setbox0=\hbox{#1}{\ooalign{\hidewidth
  \lower1.5ex\hbox{`}\hidewidth\crcr\unhbox0}}} \def\cprime{$'$}
\providecommand{\bysame}{\leavevmode\hbox to3em{\hrulefill}\thinspace}
\providecommand{\MR}{\relax\ifhmode\unskip\space\fi MR }
\providecommand{\MRhref}[2]{%
  \href{http://www.ams.org/mathscinet-getitem?mr=#1}{#2}
}
\providecommand{\href}[2]{#2}

\end{document}